\documentclass[12pt, twoside]{article}
\usepackage[all, cmtip]{xy}

\usepackage[colorlinks=true]{hyperref}

\usepackage{amsmath,amsthm,amssymb}
\usepackage{times}

\usepackage{times,upref,eucal, mathrsfs, xcolor, dsfont}
\usepackage{amsfonts,amsmath,amstext,amsbsy,amssymb, amsopn,amsthm}

\usepackage{enumerate}

\usepackage{url}

\usepackage{mathtools}
\usepackage{bookmark}

\theoremstyle{plain}
\newtheorem{thm}{Theorem}[section]
\newtheorem{lem}[thm]{Lemma}
\newtheorem{prop}[thm]{Proposition}
\newtheorem{cor}[thm]{Corollary}

\newcommand{\thistheoremname}{}
\newtheorem{genericthm}[thm]{\thistheoremname}

  \newtheorem*{genericthm*}{\thistheoremname}
\newenvironment{namedthm*}[1]
  {\renewcommand{\thistheoremname}{#1}%
   \begin{genericthm*}}
  {\end{genericthm*}}

\theoremstyle{definition}
\newtheorem{defn}{Definition}[section]

\newtheorem{ass}{Assumption}

\theoremstyle{remark}
\newtheorem{rem}{Remark}[section]
\newtheorem{exam}[rem]{Example}

\newtheorem*{claim}{Claim}

\newtheorem*{coms}{Comments}

\newcommand{\intt}{\int\limits}

\newcommand{\N}{\mathbb{N}}
\newcommand{\R}{\mathbb{R}}
\newcommand{\C}{\mathbb{C}}
\newcommand{\Z}{\mathbb{Z}}

\newcommand{\D}{\mathbb{D}}

\newcommand{\E}{\mathbb{E}}
\newcommand{\PP}{\mathbb{P}}

\newcommand{\F}{\mathfrak{F}}

\newcommand{\XX}{\mathcal{X}}
\newcommand{\ZZ}{\mathcal{Z}}

\newcommand{\VV}{\mathcal{V}}

\newcommand{\supp}{\mathrm{supp}}
\newcommand{\Var}{\mathrm{Var}}

\newcommand{\n}{\|}

\newcommand{\Ran}{\mathrm{Ran}}

\newcommand{\loc}{\mathrm{loc}}

\newcommand \tr {{\mathrm{tr}}}

\newcommand{\Diff}{\mathrm{Diff}}

\newcommand{\ev}{\mathrm{ev}}

\newcommand \Conf {{\mathrm {Conf}}}

\newcommand \et { \text{\,\,and\,\,}}

\newcommand \half{\frac{1}{2}}

\input xy
\xyoption{all}

\def\titlerunning#1{\gdef\titrun{#1}}
\makeatletter
\def\author#1{\gdef\autrun{\def\and{\unskip, }#1}\gdef\@author{#1}}
\def\address#1{{\def\and{\\\hspace*{18pt}}\renewcommand{\thefootnote}{}%
\footnote {#1}}%
\markboth{\autrun}{\titrun}}
\makeatother
\def\email#1{e-mail: #1}
\def\subjclass#1{{\renewcommand{\thefootnote}{}%
\footnote{\emph{Mathematics Subject Classification (2010):} #1}}}
\def\keywords#1{\par\medskip
\noindent\textbf{Keywords.} #1}

\begin{document}

\baselineskip=17pt

\titlerunning{Determinantal processes and holomorphic function spaces}

\title{Determinantal point processes associated with Hilbert spaces of holomorphic functions}

\author{Alexander I. Bufetov, Yanqi Qiu}

\date{}

\maketitle

\address{Alexander I. Bufetov: Aix-Marseille Universit{\'e}, Centrale Marseille, CNRS, I2M, UMR7373,  39 Rue F. Joliot Curie 13453, Marseille Cedex 13, France; Steklov Institute of Mathematics, Moscow, Russia; Institute for Information Transmission Problems, Moscow, Russia; National Research University Higher School of Economics, Moscow, Russia;  The Chebyshev Laboratory, Saint-Petersburg State University, Saint-Petersburg, Russia; \email{bufetov@mi.ras.ru}}

\address{Yanqi Qiu: Aix-Marseille Universit{\'e}, Centrale Marseille, CNRS, I2M, UMR7373,  39 Rue F. Joliot Curie 13453, Marseille Cedex 13, France;  \email{yqi.qiu@gmail.com}}


\begin{abstract}
We study determinantal point processes on $\C$ induced by the reproducing kernels of generalized Fock spaces as well as those on the unit disc $\D$ induced by the reproducing kernels of generalized Bergman spaces.  In the first case, we show that all reduced Palm measures {\it of the same order} are equivalent.  The Radon-Nikodym derivatives are computed explicitly using regularized multiplicative functionals. We also show that these determinantal point processes are rigid in the sense of Ghosh and Peres, hence reduced Palm measures {\it of different orders} are singular. In the second case, we show that all reduced Palm measures, {\it of all orders}, are equivalent. The Radon-Nikodym derivatives are computed using regularized multiplicative functionals associated with certain Blaschke products.   The quasi-invariance of these determinantal point processes under the group of diffeomorphisms with compact supports follows as a corollary.

\end{abstract}

\tableofcontents

\input xy
\xyoption{all}

\section{Introduction}\label{sec-intro}

\subsection{Main results}

\subsubsection{The case of $\C$}

Let $\psi: \C \rightarrow \R$ be a $C^2$-smooth function and equip the complex plane $\C$ with the measure $e^{-2\psi(z)} d\lambda(z)$, where $d\lambda$ is the Lebesgue measure. Assume that there exist positive constants $m, M > 0$ so that
\begin{align}\label{sub-h}
m \le \Delta \psi \le M,
\end{align}
where $\Delta$ is the Euclidean Laplacian . 

Denote by $\mathscr{F}_\psi$ the generalized Fock space with respect to the weight $e^{-2 \psi(z)}$ and let  $B_\psi$ be the reproducing kernel of $\mathscr{F}_\psi$, whose definition is recalled in Definition \ref{defn-fock}. The condition \eqref{sub-h} implies in particular the useful Christ \cite{Christ} pointwise estimate for the reproducing kernel $B_\psi$, see Theorem \ref{Christ} below.

By a theorem due to Macch\`i and Soshnikov \cite{DPP-M}, \cite{DPP-S} and Shirai-Takahashi \cite{ShirTaka0}, the kernel $B_\psi$ induces  a  determinantal point process, denoted by $\PP_{B_\psi}$,  on the complex plane $\C$ (with the background measure $e^{-2 \psi(z)} d\lambda(z)$). For more background on determinantal point processes, see, e.g. \cite{DPP-HKPV}, \cite{DPP-L}, \cite{DPP-S}, \cite{DPP-M} and \S 2 below.

Let  $\mathfrak{p} \in \C^\ell $ and   $\mathfrak{q}\in \C^k$ be two tuples of {\it distinct} points in $\C$. Denote by $\PP_{B_\psi}^{\mathfrak{p}}$ and  $\PP_{B_\psi}^{\mathfrak{q}}$ the reduced Palm measures of $\PP_{B_\psi}$ conditioned at $\mathfrak{p}$ and $\mathfrak{q}$ respectively. For the definition, see, e.g. \cite{Kallenberg}, here, we follow the notation  and conventions of \cite{BQS}.

Our first main result is that, under the assumption \eqref{sub-h}, Palm measures $\PP_{B_\psi}^{\mathfrak{p}}$ and $\PP_{B_\psi}^{\mathfrak{q}}$ of the same order are equivalent.  

\begin{thm}[Palm measures of the same order]\label{thm-A2}
Let $\psi$ satisfy \eqref{sub-h} and  let $\mathfrak{p}, \mathfrak{q} \in \C^\ell $ be any two tuples of distinct points in $\C$. 
Then  
\begin{itemize}
\item[1)] The limit
\begin{align*}
\Sigma_{\mathfrak{p}, \mathfrak{q}}(\ZZ): =  \lim_{ R \to \infty} & \bigg\{ \sum_{z \in \ZZ: | z | \le R }\log \left| \frac{(z-p_1) \dots (z-p_\ell)}{(z - q_1) \dots (z - q_\ell) } \right|
\\
 & -    \E_{\PP_{B_\psi}^{\mathfrak{q}}} \sum_{z \in \ZZ: | z | \le R }\log \left| \frac{(z-p_1) \dots (z-p_\ell)}{(z - q_1) \dots (z - q_\ell) } \right|\bigg\}
\end{align*}
exists for $\PP_{B_\psi}^{\mathfrak{q}}$-almost every configuration $\ZZ$ and the function  $\ZZ \rightarrow e^{2 \Sigma_{\mathfrak{p}, \mathfrak{q}}(\ZZ)}$ is integrable with respect to  $\PP_{B_\psi}^{\mathfrak{q}}$.
\item[2)] The Palm measures $\PP_{B_\psi}^{\mathfrak{p}}$ and $\PP_{B_\psi}^{\mathfrak{q}}$ are equivalent.
Moreover, for $\PP_{B_\psi}^{\mathfrak{q}}$-almost every configuration $\ZZ$, we have
\begin{align}\label{p-q-rn}
\frac{d\PP_{B_\psi}^{\mathfrak{p}}}{d    \PP_{B_\psi}^{\mathfrak{q}} } (\ZZ) =  \frac{  e^{2\Sigma_{\mathfrak{p}, \mathfrak{q}}(\ZZ)}  }{\E_{\PP_{B_\psi}^{\mathfrak{q}} } (e^{2\Sigma_{\mathfrak{p}, \mathfrak{q}}})}.
\end{align}
\end{itemize}
\end{thm}

\begin{defn}[Ghosh \cite{Ghosh-sine}, Ghosh-Peres\cite{Ghosh-rigid}]\label{defn-rig}
A point process $\PP$ on $\C$ is said to be rigid if for any bounded open set $D \subset \C$ with Lebesgue-negligible boundary $\partial D$, there exists a function $F_{D}$ defined on the set of configurations, measurable with respect to the $\sigma$-algebra generated by the family of random variables $\{\#_A: A \subset \C \setminus D \text{\, bounded  and Borel}\}$, where $\#_A$ is defined by
$$
\#_A(\ZZ) = \text{ the cardinality of the finite set $\ZZ \cap A$,}
$$ 
such that
$$
\#_{D} (\ZZ) = F_{D} (\ZZ\setminus D), \text{\, for $\PP$-almost every configuration $\ZZ$ over $\C$.}
$$

\end{defn}

\begin{prop}[Rigidity]\label{thm-rigid}
Under the assumption \eqref{sub-h}, the determinantal point process $\PP_{B_\psi}$ is rigid in the sense of Ghosh and Peres.
\end{prop}

Proposition \ref{rigid-dis} in the Appendix now implies 
\begin{cor}[Palm measures of different orders]\label{cor-d-order}
Under the assumption \eqref{sub-h}, if $\ell \ne k$, then the reduced Palm measures $\PP_{B_\psi}^{\mathfrak{p}}$ and $\PP_{B_\psi}^{\mathfrak{q}}$ are mutually singular.
\end{cor}

\begin{rem}
 In the particular case $\psi(z)  = \frac{1}{2}| z|^2$ (Ginibre point process), the results of Theorem \ref{thm-A2} and Corollary \ref{cor-d-order} were obtained in \cite{Osada-Shirai} with a different approach, where the authors used finite dimensional approximation by orthogonal polynomial ensembles. The rigidity in the case $\psi(z)  = \frac{1}{2}| z|^2$ is due to Ghosh and Peres \cite{Ghosh-rigid}, their original approach will be followed in our proof of Proposition \ref{thm-rigid}.
\end{rem}

\subsubsection{The case of $\D$}
In the case of Bergman spaces on the unit disc $\D$, the situation becomes quite different and the corresponding determinantal point processes in this case are not rigid. 

Consider a weight function $\omega:  \D\rightarrow \R^{+}$ and equip $\D$ with the measure $\omega(z)d\lambda(z)$. 
Denote by $\mathscr{B}_\omega$ the generalized Bergman space on $\D$ with respect to the weight $\omega$, and by $B_\omega$ its reproducing kernel, the definition is recalled in Definition \ref{defn-berg}. Assume moreover that $\omega$ satisfies 
\begin{align}\label{w-condition}
\int_\D (1 - | z|)^2 B_\omega(z,z) \omega(z)d\lambda(z)< \infty.
\end{align}
In \S \ref{sec-w-disk}, we will see that the condition \ref{w-condition} is satisfied for large class of weight function $\omega$ on $\D$, including most of the natural Bergman weights.

Again, by the theorem due to Macch\`i,  Soshnikov \cite{DPP-M}, \cite{DPP-S} and Shirai-Takahashi \cite{ShirTaka0}, the reproducing kernel $B_\omega$ induces a determinantal point process on $\D$ (with the background measure $\omega(z)d\lambda(z)$), which we denote by $\PP_{B_\omega}$.

Let  $\mathfrak{p} \in \D^\ell$ be an $\ell$-tuple of distinct points in $\D$ and denote by $\PP_{B_\omega}^{\mathfrak{p}}$ the reduced Palm measures of $\PP_{B_\omega}$  at $\mathfrak{p}$.

Under the assumption \eqref{w-condition}, we show, for any $\mathfrak{p} \in \D^\ell$ of distinct points in $\D$, the reduced Palm measure $\PP_{B_\omega}^{\mathfrak{p}}$ is equivalent to $\PP_{B_\omega}$. In particular, any two reduced Palm measures are equivalent. For the weight $\omega\equiv 1$, this result is due to Holroyd and Soo \cite{Soo}.

We now proceed to the statement of our main result in the case of $\D$. For an $\ell$-tuple $\mathfrak{p} = (p_1, \cdots, p_\ell)$ of distinct points in $\D$, set 
\begin{align}\label{def-bp}
 b_{\mathfrak{p}}(z)  =  \prod_{j=1}^\ell \frac{z-p_j}{1-\bar{p}_j z}.
 \end{align}
 
\begin{thm}\label{thm-B2}
Let $\omega$ be a weight such that  \eqref{w-condition} holds. Let  $\mathfrak{p} \in \D^\ell$ be an $\ell$-tuple of distinct points in $\D$.
Then 
\begin{itemize}
\item[1)]
The limit 
\begin{align}\label{S2}
S_{\mathfrak{p}}(\ZZ): = \lim_{ r \to 1^{-}}\left( \sum_{z \in \ZZ: | z | \le r }\log |b_{\mathfrak{p}}(z) | -    \E_{\PP_{B_\omega}} \sum_{z \in \ZZ: | z | \le r }\log |b_{\mathfrak{p}}(z) |\right)
\end{align}
exists for $\PP_{B_\omega}$-almost every configuration $\ZZ$ and the function  $\ZZ \rightarrow e^{2 S_{\mathfrak{p}}(\ZZ)}$ is integrable with respect to  $\PP_{B_\omega}$.  
\item[2)]
The Radon-Nikodym derivative $d \PP_{B_\omega}^{\mathfrak{p}}/d \PP_{B_\omega}$  is given by the formula: 
\begin{align}\label{RN2}
\frac{d \PP_{B_\omega}^{\mathfrak{p}}}{ d \PP_{B_\omega}} (\ZZ) =   \frac{e^{2 S_{\mathfrak{p}}(\ZZ)}}{\E_{\PP_{B_\omega}}    (e^{2 S_{\mathfrak{p}}})}, \text{\, for $\PP_{B_\omega}$-almost every configuration $\ZZ$.}
\end{align}
\end{itemize}
 \end{thm}
 
Theorem \ref{thm-B2} will be obtained from
 \begin{prop}\label{prop-B}
Let $\omega$ be a weight such that  \eqref{w-condition} holds. Let  $\mathfrak{p} \in \D^\ell$ and $\mathfrak{q} \in \D^k$ be two tuples of distinct points in $\D$.
 Then the Radon-Nikodym derivative $d \PP_{B_\omega}^{\mathfrak{p}}/d \PP_{B_\omega}^{\mathfrak{q}}$  is given by 
 \begin{align}\label{RN2}
\frac{d \PP_{B_\omega}^{\mathfrak{p}}}{ d \PP_{B_\omega}^{\mathfrak{q}}} (\ZZ)  =   \frac{e^{2 S_{\mathfrak{p}, \mathfrak{q}}(\ZZ)}}{\E_{\PP_{B_\omega}^{\mathfrak{q}}}    (e^{2 S_{\mathfrak{p}, \mathfrak{q}}})}, \text{\, for $\PP_{B_\omega}^{\mathfrak{q}}$-almost every configuration $\ZZ$,}
\end{align}
where $S_{\mathfrak{p}, \mathfrak{q}}(\ZZ)$ is defined for $\PP_{B_\omega}^{\mathfrak{q}}$-almost every configuration $\ZZ$, given by
\begin{align}\label{S2}
S_{\mathfrak{p}, \mathfrak{q}}(\ZZ): = \lim_{ r \to 1^{-}}\left( \sum_{z \in \ZZ: | z | \le r }\log |b_{\mathfrak{p}}(z) b_{\mathfrak{q}}(z)^{-1}| -    \E_{\PP_{B_\omega}^{\mathfrak{q}}} \sum_{z \in \ZZ: | z | \le r }\log |b_{\mathfrak{p}}(z) b_{\mathfrak{q}}(z)^{-1}|\right).
\end{align}
 \end{prop}

\begin{rem}

If $\psi$ (resp. $\omega$) is a radial function, then the monomials $(z^n)_{n\ge 0}$ are orthogonal in the corresponding Hilbert space, hence the determinantal point process $\PP_{B_\psi}$ (resp. $\PP_{B_\omega}$) can be naturally approximated by {\it orthogonal polynomial ensembles}. In particular, if $\psi(z) = \frac{1}{2}|z|^2$ for all $z\in \C$, then $\PP_{B_\psi}$ is the Ginibre point process, see chapter 15 of Mehta's book \cite{Mehta-book}; if  $\omega(z) \equiv 1$ for all $z \in \D$, then $\PP_{B_\omega}$ is the determinantal point process describing the zero set of a Gaussian analytic function on the hyperbolic disc $\D$, see \cite{PV-acta}. Our study, however, goes beyond the radial setting and our methods work for more general phase spaces as well. 
 \end{rem}
 
 \begin{rem}
The regularized multiplicative functionals are necessary in Theorem \ref{thm-A2}, Theorem \ref{thm-B2} and Proposition \ref{prop-B}: indeed, when $\omega \equiv 1$,  for $\PP_{B_\omega}$-almost every configuration $\ZZ$ on $\D$,   the points in the configuration $\ZZ$ violate the Blaschke condition: 
\begin{align}\label{no-bla}
\sum_{z\in \ZZ} (1 - | z |) = \infty,
\end{align}
whence for any $\mathfrak{p}\in \D^\ell$,  we have,
 \begin{align}\label{violate-blaschke}
\prod_{z\in \ZZ} | b_{\mathfrak{p}}(z)| = 0, \text{\, for $\PP_{B_\omega}$-almost every configuration $\ZZ$,}
\end{align}
so the simple multiplicative functional is identically $0$. 
To see \eqref{no-bla}, we use the Kolmogorov three-series theorem and  the fact (Peres and Vir\'ag \cite{PV-acta}) that, for $\PP_{B_\omega}$-distributed random configurations $\ZZ$, the set of moduli $\{|z|: z \in \ZZ\}$ has same law as the set of random variables   $\{U_k^{1/(2k)}\}$, where $U_1, U_2, \dots $ are independent identically distributed  random variables such that $U_1$ has a uniform distribution in $[0,1]$. A direct computation shows that 
$$
 \E_{\PP_{B_\omega}}\sum_{z\in \ZZ} (1 - | z |) = \sum_{k } (1 - \E(U_k^{1/(2k)}))  = \infty.
$$

The determinantal point process $\PP_{B_\omega}$ in the case $\omega\equiv1$ describes the zero set of a Gaussian analytic function on $\D$:
$$
F_\D(z)= \sum_{n=0}^\infty g_n z^n,
$$
where  $(g_n)_{n\ge0}$ is a sequence of  independent identically distributed standard complex Gaussian random variables. Direct computation shows that 
$$
\E \n F_\D\n_{H^2}^2 = \infty \et \E\n F_\D\n_{B_\omega}^2  = \infty,
$$
hence the random holomorphic function almost surely belongs neither to the Hardy space $H^2$ nor to the Bergman space, thus it is  not surprising that the zero set of $F_\D$ almost surely violates Blaschke condition.

\end{rem}

\subsection{Quasi-invariance}

 Let $U = \C$ or $\D$. Let  $F: U\rightarrow U$ be a diffeomorphism. Its support, denoted by $\supp(F)$, is defined as the {\it relative closure} in $U$ of the subset $\{ z \in U: F (z) \ne z\}$. The totality of diffeomorphisms with compact supports is a group denoted by  $\Diff_c(U)$, i.e., 
  $$
 \Diff_c(U):  = \left\{F: U \rightarrow U\Big| \text{$F$ is a diffeomorphism and $\supp(F)$ is compact}\right\}.
 $$
The group $\Diff_c(U)$ naturally acts on the set of configurations on $U$:  given any diffeomorphism $F\in \Diff_c(U)$ and any configuration $\ZZ$ on $U$, 
$$
(F, \mathcal{Z}) \mapsto F(\ZZ) : = \{ F(z): z \in \mathcal{Z}\}.
$$
 
Recall that the Jacobian $J_F$ of the function $F: U\rightarrow U$ is defined by 
$$
J_F(z) = | \det D F(z)|.
$$

\begin{cor}\label{quasi-inv}
Let $\PP_K$ be a determinantal point process on $U$, which is either the determinantal point process $\PP_{B_\psi}$ on $\C$ or the determinantal point process $\PP_{B_\omega}$ on $\D$. Then under Assumption \eqref{sub-h} in the case of $\C$ or, in the case of $\D$ Assumption \eqref{w-condition}, $\PP_K$ is quasi-invariant under the induced action of the group $\Diff_c(U)$. 

More precisely, let $F \in \Diff_c(U)$ and  let $V \subset U$ be any precompact subset containing $\supp (F)$.  For $\mathbb{P}_{K}$-almost every configuration $\ZZ $ the following holds: if $\ZZ \bigcap V  = \{q_1, \dots, q_\ell\}$, then
\begin{align*}
\frac{d \mathbb{P}_K \circ F}{ d \mathbb{P}_{K}}  (\ZZ) = &  \frac{\det [ K ( F(q_i), F(q_j))]_{i, j = 1}^\ell}{\det [ K (q_i, q_j)]_{i, j = 1}^\ell} \cdot  \frac{d\PP_{K}^{\mathfrak{p}}}{ d \PP_K^{\mathfrak{q}}}(\ZZ)  \cdot  \prod_{i = 1}^\ell  J_F (  q_i ) ,
\end{align*} 
where   $\mathfrak{q}  = (q_1, \dots, q_\ell)$ and $\mathfrak{p} = (F(q_1), \dots, F(q_\ell))$.
\end{cor}

\begin{proof}
This  is an immediate consequence of Theorem \ref{thm-A2}, Proposition \ref{prop-B} and \cite[Prop. 2.19]{BQS}.
\end{proof}

\begin{rem}
Grigori Olshanski in \cite{OQS}   has shown that the determinantal point process on $\Z$
governed by  the Gamma kernel is quasi-invariant under the group of finite permutations of $\Z$ and has expressed the Radon-Nikodym derivative as a generalized multiplicative functional.
In \cite{BQS}  quasi-invariance under the infinite symmetric group is established for a large class of determinantal measures on $\Z$ and it is also shown that a large class of
determinantal measures on $\R$ is quasi-invariant under the group of diffeomorphisms with compact support. Quasi-invariance under local deformations of the phase space can be seen as a weak form of exchangeability and, thus, a measure of chaos of our processes. For example, Gibbs measures are quasi-invariant under local perturbations, and the Radon-Nikodym derivative is a multiplicative functional. As Ghosh-Peres rigidity shows, particles of a determinantal process interact much more strongly than those in a Gibbs field. The quasi-invariance can nonetheless be seen as the analogue, in our situation, of the Gibbs property.  In the sequel \cite{BQ-cond} to this paper, quasi-invariance is used
in order to compute, for determinantal point processes corresponding to generalized Fock spaces, the
conditional measure in a bounded domain with respect to the configuration in the complement. This conditional measure is proved to be an orthogonal polynomial ensemble whose weight is found explicitly.
\end{rem}

\subsection{Unified approach for obtaining Radon-Nikodym derivatives}
In this section, let us describe briefly the main idea of our unified approach for obtaining the Radon-Nikodym derivatives in Theorem \ref{thm-A2}, Theorem \ref{thm-B2} and Proposition \ref{prop-B}.  

\subsubsection{Relations between Palm subspaces}
If $\mathfrak{p} \in \C^\ell$ is an $\ell$-tuple of distinct points of $\C$, we define the {\it Palm subspace}:
\begin{align}\label{Fock-palm-sub}
\mathscr{F}_\psi(\mathfrak{p})  :=  \left\{ \varphi \in \mathscr{F}_\psi:  \varphi(p_1) = \cdots = \varphi(p_\ell) = 0 \right\}.
\end{align}
Let $B_\psi^{\mathfrak{p}}$ denote the reproducing kernel of $\mathscr{F}_\psi(\mathfrak{p})$. 

Similarly, if $\mathfrak{p} \in \D^\ell$ is an $\ell$-tuple of distinct points of $\D$, we define the Palm subspace
\begin{align}\label{ber-palm-sub}
\mathscr{B}_\omega(\mathfrak{p})  =  \left\{ \varphi \in \mathscr{B}_\omega:  \varphi(p_1) = \cdots = \varphi(p_\ell) = 0 \right\},
\end{align}
and denote its reproducing kernel by $B_\omega^{\mathfrak{p}}$.

By Shirai-Takahashi's theorem, which motivates our terminology, see Theorem \ref{ST} below,  these Palm subspaces are related to the reduced Palm measures: $B_\psi^{\mathfrak{p}}$ (resp. $B_\omega^{\mathfrak{p}}$) is the correlation kernel of $\PP_{B_\psi}^{\mathfrak{p}}$ (resp. $\PP_{B_\omega}^{\mathfrak{p}}$), i.e., we have
$$
\PP_{B_\psi}^{\mathfrak{p}} = \PP_{B_\psi^{\mathfrak{p}}} \,\, (\text{resp. } \PP_{B_\omega}^{\mathfrak{p}} = \PP_{B_\omega^{\mathfrak{p}}}).
$$

In what follows, for a measured space $(E, \mu)$, a Borel function $g: E \rightarrow \C$ and a certain subspace $L\subset L^2(E, \mu)$, we denote by $gL$ the space defined by 
\begin{align}\label{mul-f-space}
gL : =\{ gf| f \in L\}. 
\end{align}
Note that in the above definition, even if $L$ is closed and $gL\subset L^2(E, \mu)$, in general, we do not require $gL$ to be closed in  $L^2(E, \mu)$.

\begin{prop}\label{prop-F-rel}
For any pair of $\ell$-tuples $\mathfrak{p}, \mathfrak{q}\in \C^\ell$ of distinct points in $\C$, we have
\begin{align}\label{F-rel}
\mathscr{F}_\psi(\mathfrak{p}) =  \frac{(z-p_1) \cdots (z-p_\ell)}{ (z-q_1)\cdots (z-q_\ell)} \cdot  \mathscr{F}_\psi(\mathfrak{q}),
\end{align}
the equality is understood as in the definition \eqref{mul-f-space}. 
\end{prop}
 
\begin{prop}\label{prop-B-rel}
Let $k, \ell \in \N\cup \{0\}$ and let  $\mathfrak{p} \in \D^\ell, \mathfrak{q}\in \D^k $ be two tuples of distinct points in $\D$.  Then 
 \begin{align}\label{B-rel}
 \mathscr{B}_\omega(\mathfrak{p})=   \prod_{j=1}^\ell \frac{z-p_j}{1-\bar{p}_j z} \left( \prod_{j=1}^k \frac{z-q_j}{1-\bar{q}_j z}\right)^{-1}  \cdot \mathscr{B}_\omega(\mathfrak{q}).
 \end{align}
 In particular, we have
 $$
 \mathscr{B}_\omega(\mathfrak{p})= \prod_{j=1}^\ell \frac{z-p_j}{1-\bar{p}_j z} \cdot \mathscr{B}_\omega.
 $$ 
\end{prop}

\begin{coms}
\begin{itemize}
\item The proofs of Propositions \ref{prop-F-rel} and \ref{prop-B-rel} are immediate from the definitions \eqref{Fock-palm-sub} and \eqref{ber-palm-sub} and basic properties of holomorphic functions.  For instance, by symmetry, for proving \eqref{F-rel}, it suffices to prove that 
\begin{align}\label{F-rel-pf}
  \frac{(z-p_1) \cdots (z-p_\ell)}{ (z-q_1)\cdots (z-q_\ell)} \cdot  \mathscr{F}_\psi(\mathfrak{q})\subset \mathscr{F}_\psi(\mathfrak{p}). 
\end{align}
But if $f \in \mathscr{F}_\psi(\mathfrak{q})$, then, by definition, $f$ is holomorphic on $\C$ and vanishes at $q_1, \cdots q_\ell$, hence the function $ \frac{(z-p_1) \cdots (z-p_\ell)}{ (z-q_1)\cdots (z-q_\ell)} \cdot f$ is holomorphic on $\C$ and vanishes at $p_1, \cdots, p_\ell$. For finishing the proof of \eqref{F-rel-pf}, it remains to prove that 
$$
\int_\C \left|\frac{(z-p_1) \cdots (z-p_\ell)}{ (z-q_1)\cdots (z-q_\ell)} \cdot f (z) \right| ^2 e^{-2 \psi(z)} d\lambda(z) < \infty.
$$
But this follows immediately from the following inequality 
\begin{align*}
& \int_\C \left|\frac{(z-p_1) \cdots (z-p_\ell)}{ (z-q_1)\cdots (z-q_\ell)} \cdot f (z) \right| ^2 e^{-2 \psi(z)} d\lambda(z)
\\
 \le & \int_{\{|z| \le R \}} \left|\frac{(z-p_1) \cdots (z-p_\ell)}{ (z-q_1)\cdots (z-q_\ell)} \cdot f (z) \right| ^2 e^{-2 \psi(z)} d\lambda(z) 
 \\
  & + K_R \int_{\{|z| > R \}} \left|f (z) \right| ^2 e^{-2 \psi(z)} d\lambda(z),
\end{align*}
where  $R = 1+ \max_{1\le i \le \ell}| q_i| $ and $K_R =\sup_{|z|> R}\left|\frac{(z-p_1) \cdots (z-p_\ell)}{ (z-q_1)\cdots (z-q_\ell)}\right|^2< \infty $.  The equality \eqref{B-rel} can be proved similarly. 
\item A common feature,  naturally, needed later,  of Propositions \ref{prop-F-rel} and \ref{prop-B-rel},  is shown by the following relations
\begin{align}\label{convergence-rate}
\lim_{| z | \to \infty} \left|   \frac{(z-p_1) \cdots (z-p_\ell)}{ (z-q_1)\cdots (z-q_\ell)}\right| = 1  \et \lim_{|z| \to 1^{-}} \left|\prod_{j=1}^\ell \frac{z-p_j}{1-\bar{p}_j z} \right| = 1.
\end{align}
The rate of convergence in \eqref{convergence-rate} also plays an important r\^ole for defining the regularized multiplicative functionals, see \S \ref{derivation-fock} and \S \ref{derivation-berg}.
\end{itemize}

\end{coms}

\subsubsection{Radon-Nikodym derivatives as regularized multiplicative functionals}

For obtaining the Radon-Nikodym derivatives in question, we  develop  in Theorem \ref{thm-C} a general result on  regularized multiplicative functionals. This most technical result of the paper, an extension of  \cite[Prop. 4.2]{BQS} (cf. Proposition \ref{prop-BQS} below),  is, we hope, interesting in its own right; the stronger statement is also necessary  for our argument in the case of $\C$, in which the main result in  \cite{BQS}  is not applicable.  The difference is that instead of Hilbert-Schmidt operators used in \cite{BQS}, here we must work with the von Neumann-Schatten class of order three; see section 4 below for details.

By Theorem \ref{thm-C}, under the assumption \eqref{sub-h} on $\psi$,  we can show that  the regularized multiplicative functional, i.e., the formula \eqref{RN2},  is well-defined. 
This regularized multiplicative functional is then shown to be exactly the Radon-Nikodym derivative between the desired reduced Palm measures  of the same order for the determinantal point process $\PP_{B_\psi}$.   

The regularized multiplicative functionals in the case of $\D$ are technically simpler and the full force of Theorem \ref{thm-C} is not needed.

\subsection{Organization of the paper}

The paper is organized as follows. The basic material in the theory of determinantal point processes is recalled in \S \ref{sec-DPP}. The definitions concerning generalized Fock spaces and generalized Bergman spaces are given in \S \ref{sec-bergman}. In \S \ref{S-reg-m} we define  {\it regularized multiplicative functionals} which play the main r\^ole in the proof and  state the technical Theorem \ref{thm-C}. Theorem \ref{thm-C}  is then applied to    determinantal point processes associated with generalized Fock spaces in \S \ref{sec-case-fock} and to  those associated with  generalized Bergman spaces in \S \ref{sec-case-bergman}. The subsequent \S \ref{proof-thm-C} is devoted to the proof of Theorem \ref{thm-C}. A general proposition
showing  that if a point process is rigid in the sense of Ghosh and Peres then its Palm measures of different orders are singular is proved in the Appendix (\S \ref{sec-app}).

 \begin{rem}
 Part of our main results in this paper were announced in \cite{BQ-cras}. 
 \end{rem}

\section{Spaces of configurations and determinantal point processes}\label{sec-DPP}

Let $E$ be a locally compact complete separable metric space equipped with a sigma-finite Borel measure $\mu$. The space $E$ will be later referred to as {\it phase space}. The measure $\mu$ is referred to as  {\it reference measure} or {\it background measure}.  By a configuration $ \mathcal{X}$ on the phase space $E$, we mean a locally finite subset of $ \mathcal{X} \subset E$.  Identify a configuration $\mathcal{X} \in \Conf(E)$ with  the Radon measure
$$
m_\mathcal{X} : = \sum_{x \in \mathcal{X}} \delta_x,
$$
where $\delta_x$ is the Dirac mass on the point $x$.  The space of configurations $\Conf(E)$ is then identified with a subset of the space $\mathfrak{M}(E)$ of Radon measures on $E$ and becomes itself a complete separable metric space. The space  $\Conf(E)$ is naturally equipped  with its Borel sigma algebra.

Points in a configuration will also be called particles.  In this paper, the italicized letters as $\XX, \mathcal{Y}, \ZZ$  always denote configurations.

\subsection{Additive functionals and multiplicative functionals} We recall the definitions of additive and multiplicative functionals on the space of configurations.

If $\varphi: E \rightarrow \C$ is a measurable function on $E$, then the additive functional (which is also called linear statistic)  $S_\varphi: \Conf(E) \rightarrow \C$ 
 corresponding to $\varphi$ is defined by 
 $$
 S_\varphi(\mathcal{X}) = \sum_{x \in \mathcal{X}} \varphi(x)
 $$
 provided the sum  $\sum_{x \in \mathcal{X}} \varphi(x)$ converges absolutely. If the sum  $\sum_{x \in \mathcal{X}} \varphi(x)$ fails to converge absolutely, then the additive functional is not defined at $\mathcal{X}$.

Similarly, the multiplicative functional $\Psi_g: \Conf(E) \rightarrow [0, \infty]$ associated with  a non-negative measurable function $g: E\rightarrow \R^{+}$,  is  defined as the function
\begin{align*}
\Psi_g(\mathcal{X}): = \prod_{x \in \mathcal{X}} g(x),
\end{align*}
provided the product $\prod\limits_{x \in \mathcal{X}} g(x)$ absolutely converges to a value in $[0, \infty]$. If the product $\prod\limits_{x \in \mathcal{X}} g(x) $ fails to converge absolutely, then the multiplicative functional is not defined at the configuration $\mathcal{X}$.

\subsection{Locally trace class operators and their kernels} 
Let $L^2(E, \mu)$ denote the complex Hilbert space of $\C$-valued square integrable functions on $E$. Let $\mathscr{S}_1(E, \mu)$ be the space of trace class operators on $L^2(E, \mu)$ equipped with the trace class norm $\n \cdot \n_{\mathscr{S}_1}$.   Let $\mathscr{S}_{1, \loc}(E, \mu)$ be the space of locally trace class operators, that is, the space of bounded operators $K: L^2(E, \mu) \rightarrow L^2(E, \mu)$ such that for any bounded subset $B \subset E$,  we have 
$$
\chi_B K \chi_B \in \mathscr{S}_1(E, \mu).
$$

A locally trace class operator $K$ admits a kernel, for which we use the same symbol $K$.  In this paper, we are especially interested in locally trace class orthogonal projection operators.  Let, therefore, $\Pi \in \mathscr{S}_{1, \loc}$ be an operator of orthogonal projection onto a closed subspace $L \subset L^2(E, \mu)$.  All kernels considered in this paper are supposed to satisfy the following

\begin{ass}\label{Asump1} There exists a subset $\widetilde{E} \subset E$, satisfying $\mu(E\setminus \widetilde{E}) = 0$ such that 
\begin{itemize}
\item For any $q \in \widetilde{E}$, the function $h_q(\cdot) = \Pi(\cdot, q) $ lies in $L^2(E, \mu)$ and for any $f \in L^2(E, \mu)$, we have
$$
(\Pi f)(q) = \langle f, h_q \rangle_{L^2(E, \mu)}.
$$   
In particular, if $f$ is a function in $L$, then by letting  $f(q) = \langle f, h_q\rangle_{L^2(E, \mu)}$, for any $q\in \widetilde{E}$, the function $f$ is defined everywhere on $\widetilde{E}$ (which is slightly stronger than almost everywhere defined on $E$). 
\item The diagonal values $\Pi(q,q)$ of the kernel $\Pi$ are defined for all $q \in \widetilde{E}$  and we have $\Pi(q,q) = \langle h_q, h_q\rangle_{L^2(E,\mu)}$. Moreover, for any bounded Borel subset $B\subset E$,  
$$
\tr(\chi_B \Pi \chi_B) = \intt_B \Pi(x,x) d\mu(x).
$$
\end{itemize} 
\end{ass}

\subsection{Definition of determinantal point processes} A Borel probability $\PP$ on $\Conf(E)$ will be called a point process on $E$.  Recall that the point process $\PP$ is said to admit $k$-th correlation measure $\rho_k$ on $E^k$ if for any continuous compactly supported function $\varphi: E^k \rightarrow \C$, we have 
\begin{align*}
\int\limits_{\Conf(E)} \sum_{x_1, \dots, x_k \in \XX}^*  \varphi(x_1, \dots, x_k) \PP(d \mathcal{X}) =   \intt_{E^k} \varphi(q_1, \dots, q_k) d \rho_k (q_1, \dots, q_k),
\end{align*}
where $\sum\limits^{*}$ denotes the sum over all ordered $k$-tuples of {\it distinct} points $(x_1, \dots, x_k) \in \mathcal{X}^k$. 

Given a bounded measurable subset $A \subset E$, we define $\#_A: \Conf(E) \rightarrow \N\cup\{0\}$ by 
$$
\#_A(\mathcal{X})= \text{ the number of particles in $\mathcal{X} \cap A $.}
$$
Then the point process $\PP$ is determined by the joint distributions of $\#_{A_1}, \dots, \#_{A_n}$, if $A_1, \dots, A_n$ range over the family of bounded measurable subsets of $E$.

A Borel probability measure $\PP$ on $\Conf(E)$ is called determinantal if there exists an operator $K\in \mathscr{S}_{1, \loc}(E, \mu)$ such that for any bounded measurable function $g$, for which $g-1$ is supported in a bounded set $B$, we have
\begin{align}\label{DPP}
\E_{\PP} \Psi_g = \det \left( 1 + (g - 1) K \chi_B \right).
\end{align}
The Fredholm determinant is well-defined since $(g-1) K \chi_B \in \mathscr{S}_{1}(E, \mu)$. The equation \eqref{DPP} determines the measure $\PP$ uniquely and we will denote it by $\PP_K$ and the kernel $K$ is said to be  {\it a correlation kernel} of the determinantal point process $\PP_K$. Note that $\PP_K$ is uniquely determined by $K$, but different kernels may yield the same point process.  

By a theorem due to Macch\`i and Soshnikov  \cite{DPP-M}, \cite{DPP-S}  and Shirai-Takahashi  \cite{ShirTaka0}, any Hermitian positive contraction in $\mathscr{S}_{1, \loc}(E,\mu)$ defines a determinantal point process. In particular, the projection operator on a {\it reproducing kernel Hilbert space} induces a determinantal point process.

\begin{rem}\label{rem-unit}
If  $\alpha: E \rightarrow \C$ is a Borel function such that $| \alpha(x) |  = 1$ for $\mu$-almost every $x\in E$, and if $\Pi \in \mathscr{S}_{1, \loc}$ is the operator of orthogonal projection onto a closed subspace $L \subset L^2(E, \mu)$, then $\Pi$ and $\alpha \Pi \overline{\alpha}$ define the same determinantal point process, i.e., $$\PP_{\alpha \Pi \overline{\alpha}} = \PP_\Pi.$$ Note that $\alpha \Pi \overline{\alpha}$ is the orthogonal projection onto the subspace $\alpha(x) L$.
\end{rem}

\subsection{Palm measures and Palm subspaces} 
In this paper, by Palm measures, we always mean {\it reduced} Palm measures. We refer to \cite{Kallenberg}, \cite{Daley-Vere} for more details on Palm measures of general point processes.

Let $\PP$ be a point process on $\Conf(E)$. Assume that $\PP$ admits $k$-th correlation measure $\rho_k$ on $E^k$. Then for $\rho_k$-almost every $\mathfrak{q} = (q_1, \dots, q_k) \in E^k$ of distinct points in $E$, one can define a point process  on $E$,  denoted by $\PP^{\mathfrak{q}}$ and is called (reduced) Palm measure of $\PP$ conditioned at $\mathfrak{q}$, by the following disintegration formula: for any non-negative Borel test function $u: \Conf(E) \times E^k\rightarrow \R $, 
\begin{align}\label{def-Palm}
\int\limits_{\Conf(E)}  \sum_{q_1, \dots, q_k \in \XX}^{*} u(\XX; \mathfrak{q}) \PP(d \XX)  =    \int\limits_{E^k} \rho_k(d\mathfrak{q}) \!\int\limits_{\Conf(E)} \!  u (\XX \cup \{q_1, \dots, q_k\};  \mathfrak{q})  \PP^{\mathfrak{q}}(d\XX),
\end{align}
where $\sum\limits^{*}$ denotes the sum over all mutually distinct points $q_1, \dots, q_k \in \XX$. 

Informally, $\PP^{\mathfrak{q}}$ is the conditional distribution of $\mathcal{X} \setminus \{q_1, \dots, q_k\}$ on $\Conf(E)$ conditioned to the event that all particles $q_1, \dots, q_k $ are in the configuration $\mathcal{X}$, provided that $\mathcal{X}$ has distribution $\PP$.

Now let $\PP_\Pi$ be a determinantal point process on $\Conf(E)$ induced by the projection operator $\Pi$. Let $\mathfrak{q}= (q_1, \dots, q_k) \in \widetilde{E}^k$ be a $k$-tuple  of distinct points in $\widetilde{E}\subset E$, where $\widetilde{E}$ is as in Assumption \ref{Asump1}. Set 
\begin{align}\label{palm-sub}
L(\mathfrak{q})  = \{\varphi \in L : \varphi (q_1) = \cdots = \varphi(q_k) = 0\}.
\end{align}
The space $L(\mathfrak{q})$ will be called the {\it Palm subspace} of $L^2(E, \mu)$ corresponding to $\mathfrak{q}$. Both the operator of orthogonal projection from $L^2(E, \mu)$ onto the subspace $L(\mathfrak{q})$ and the reproducing kernel of $L(\mathfrak{q})$  will be denoted by $\Pi^{\mathfrak{q}}$.

Explicit formulae for $\Pi^{\mathfrak{q}}$ in terms of the kernel $\Pi$ are known, see Shirai-Takahashi \cite{ST-Palm}. Here we recall that for a single point  $q\in \widetilde{E}$, we have 
\begin{align}\label{1-rank}
\Pi^q(x,y) = \Pi(x,y) - \frac{\Pi(x,q) \Pi(q,y)}{\Pi(q,q)}.
\end{align}
If  $\Pi(q,q)= 0$, we set $\Pi^q = \Pi$. In general, we have the iteration
$$
\Pi^{\mathfrak{q}} = (\cdots (\Pi^{q_1})^{q_2}\cdots)^{q_k}.
$$
Note that the order of the points $q_1, q_2, \cdots q_k$ has no effect in the above iteration.

\begin{thm}[Shirai and Takahashi \cite{ST-Palm}]\label{ST}
For any $k \in \N$ and for $\rho_k$-almost every $k$-tuple $\mathfrak{q}  \in E^k$ of distinct points in $E$, the Palm measure $\PP_\Pi^{\mathfrak{q}}$ is induced by the kernel $\Pi^{\mathfrak{q}}$: 
$$
\PP_\Pi^{\mathfrak{q}} = \PP_{\Pi^{\mathfrak{q}}}.
$$
\end{thm}

\subsection{Rigidity}
Let $\PP$ be a point process over $\C$. We will use the following result on the rigidity of point processes (see Definition \ref{defn-rig}).
\begin{thm}[Ghosh \cite{Ghosh-sine}, Ghosh and Peres \cite{Ghosh-rigid}]\label{thm-GP}
Let $\PP$ be a point process on $\C$ whose first correlation measure $\rho_1$ is absolutely continuous with respect to the Lebesgue measure. Suppose that for any $R>0$ and $0< \varepsilon< 1$, there exists a $C_c^2$-smooth function $\Phi_{\varepsilon, R}$ such that $\Phi_{\varepsilon, R}(z) = 1$ on $\{z\in\C: |z|\le R\}$ and $\Var_{\PP} (S_{\Phi_{\varepsilon,R}}) < \varepsilon.$ Then the point process $\PP$ is rigid.
\end{thm}

The reader is referred also to \cite{Buf-rigid} \cite{BDQ} for more results on rigidity of point processes. 

\section{Generalized Fock spaces and Bergman spaces}\label{sec-bergman}
Let $\mathscr{O}(\C)$ and $\mathscr{O}(\D)$ denote the space of holomorphic functions on the whole plane $\C$ and on the unit disk $\D$ respectively.

Let $\psi: \C \rightarrow \R$ be a function satisfying the assumption \eqref{sub-h} and denote 
$$
dv_\psi(z) = e^{- 2\psi(z)} d\lambda(z),
$$
where $d \lambda$ is the Lebesgue measure on $\C$. 
\begin{defn}\label{defn-fock}
If the linear subspace 
 $$
 \mathscr{F}_\psi: =L^2(\C, dv_\psi)  \cap \mathscr{O}(\C)
 $$
is closed in $L^2(\C, dv_\psi)$, then it will called generalized Fock space with respect to the measure $dv_\psi$.  The orthogonal projection $P: L^2(dv_\psi) \rightarrow \mathscr{F}_\psi$ is  given by integration against a reproducing kernel $B_\psi(z,w)$ (analytic in $z$ and anti-analytic in $w$): 
\begin{align}\label{reproducing}
(P f)(z) = \int_\C  f(w) B_\psi(z,w) e^{-2 \psi(w)} d\lambda(w).
\end{align}
\end{defn}

\begin{defn}\label{defn-berg}
Let $\D \subset \C$ be the open unit disc.  A weight function $\omega:  \D\rightarrow \R^{+}$ is called a {\it Bergman weight}, if it is integrable with respect to the Lebesgue measure and  the generalized Bergman space 
$$
\mathscr{B}_\omega :  =  L^2(\D, \omega d\lambda) \cap \mathscr{O}(\D)
$$ is closed in $ L^2(\D, \omega d\lambda)$ and the evaluation functionals $f \rightarrow f(z)$ on $\mathscr{B}_\omega$ are uniformly bounded on any compact subset of $\D$. In such situation, the space  $\mathscr{B}_\omega$ is a reproducing kernel Hilbert space, its reproducing kernel will be denoted as $B_\omega$.
\end{defn}

\bigskip

We shall need Christ's pointwise estimate (cf. \cite{Christ}, \cite{Delin}, \cite{Schuster}) of the reproducing kernel $B_\psi(z,w)$. Theorem 3.2 in \cite{Schuster} gives the estimate in the form most convenient for us. 

\begin{thm}[Christ]\label{Christ}
Let $\psi\in C^2(\C)$ be a real-valued function satisfying \eqref{sub-h}. Then there are contants $\delta, C>0$ such that for all $z, w \in \C$, 
\begin{align}\label{pt-es}
| B_\psi(z,w)|^2 e^{-2 \psi(z) - 2 \psi(w)} \le C e^{- \delta | z - w|}.
\end{align}
In particular, for all $z\in \C$, 
\begin{align}\label{d-es}
B_\psi(z,z)e^{-2 \psi (z) } \le C.
\end{align}
\end{thm}

\begin{rem}
For the Gaussian case $\psi (z) = \frac{1}{2}| z|^2$, we have the following  explicit formula 
$$
| B_{\psi}(z,w)|^2 e^{-2 \psi(z) - 2 \psi(w)} = \pi^{-2} e^{- | z - w|^2}.
$$
\end{rem}

\section{Regularized multiplicative functionals}\label{S-reg-m}

\subsection{Statement of the main result}
As \eqref{violate-blaschke} shows, simple multiplicative functionals cannot be used in our situation. Following \cite{BQS}, we  use  regularized multiplicative functionals whose definitions we now recall. 

Let $f : E \rightarrow \C$ be a Borel function. Set 
\begin{align}\label{var-pi-f}
\Var(\Pi, f)  = \frac{1}{2} \iint_{E^2} | f(x)- f(y)|^2 | \Pi(x,y)|^2 d\mu(x) d\mu(y).
\end{align}
Introduce the Hilbert space $\VV(\Pi)$ in the following way: the elements of $\VV(\Pi)$ are functions $f$ on $E$ satisfying $\Var(\Pi, f) < \infty$; functions that differ by a constant are identified. The square of the norm of an element $f\in \VV(\Pi)$ is precisely $\Var(\Pi, f)$. 

Let $S_f: \Conf(E)\rightarrow \C$  be the corresponding additive functional, such that $S_f\in L^1(\Conf(E), \PP_\Pi)$.  Set 
\begin{align}\label{add}
\overline{S}_f = S_f  - \E_{\PP_{\Pi}} S_f.
\end{align}
 If, moreover, $S_f \in L^2(\Conf(E), \PP_\Pi)$, then it is easy to see that 
\begin{align}\label{var}
\E_{\PP_\Pi} | \overline{S}_f|^2 = \Var_{\PP_{\Pi}} (S_f) =  \Var(\Pi, f).
\end{align}

\begin{defn}\label{d-v0}
Let  $\VV_0(\Pi) $ be the subset of functions $f\in \VV(\Pi)$, such that there exists an exhausting sequence of bounded subsets $(E_n)_{n\ge 1}$, depending on $f$,  so that 
$$
f \chi_{E_n} \xrightarrow[n\to \infty]{\VV(\Pi)} f.
$$
\end{defn}

The identity \eqref{var} implies that there exists a unique isometric embedding (as metric spaces)
$$
\overline{S}: \VV_0(\Pi) \rightarrow L^2(\Conf(E), \PP_\Pi)
$$ extending the definition \eqref{add}, so that we have 
\begin{align}\label{def-S-f}
\overline{S}_f=\lim_{n\to \infty}  \sum_{x \in \XX \cap E_n} f(x)- \E_{\PP_\Pi}  \sum_{x \in \XX\cap E_n} f(x).
\end{align}

\begin{defn}\label{def-mf-tilde}
For a non-negative function $g: E \rightarrow \R$ such that $\log g \in \VV_0(\Pi)$ we set 
\begin{align*}
\widetilde{\Psi}_g = \exp(\overline{S}_{\log g}).
\end{align*}
If, moreover, $\widetilde{\Psi}_g \in L^1(\Conf(E), \PP_\Pi)$, then we set 
$$
\overline{\Psi}_g = \frac{\widetilde{\Psi}_g}{ \E_{\PP_\Pi} \widetilde{\Psi}_g }.
$$
\end{defn}
The function $\overline{\Psi}_g$ is called the regularized multiplicative functional associated to $g$ and $\PP_\Pi$. For specifying the dependence on $\PP_\Pi$, the notation $\overline{\Psi}_{g}^{\Pi}$ will also be used.  By definition, for $\PP_\Pi$-almost every configuration $\XX$, the following identity holds:
\begin{align}\label{finite-app}
\log \overline{\Psi}_g^\Pi(\XX) = \lim_{n\to \infty}  \sum_{x \in \XX \cap E_n} \log g(x)  - \E_{\PP_\Pi}  \left(\sum_{x \in \XX \cap E_n} \log g(x)\right).
\end{align}
 
Clearly, $\overline{\Psi}_g^\Pi$ is a probability density  for $\PP_\Pi$, since $\E_{\PP_\Pi} (\overline{\Psi}_g^\Pi) = 1$.

\begin{thm}\label{thm-C}
Let $g$ be a nonnegative Borel function on $E$  such that it is positive up to a $\mu$-negligible set and for any $\varepsilon>0$ the subset $\{x \in E: | g(x) - 1| \ge \varepsilon \} $ is  bounded. Assume moreover that there exists an increasing sequence of bounded subsets $(E_n)_{n \ge 1}$ exhausting the whole phase space $E$ and such that
\begin{align}\label{singularity}
\intt\limits_{E_n} | g(x)-1| \Pi(x,x)d\mu(x) < \infty;
\end{align}
\begin{align}\label{decay}
 \quad \intt\limits_{ E_n^c} | g(x)-1|^3 \Pi(x,x)d\mu(x)< \infty;
\end{align}
\begin{align}\label{variance}
\iint\limits_{E_n^c \times E_n^c} | g(x) - g(y) |^2 | \Pi(x,y)|^2 d\mu(x)d\mu(y) < \infty;
\end{align}
\begin{align}\label{flat-cut}
\lim_{n\to \infty}\tr (\chi_{E_n} \Pi | g-1|^2 \chi_{E_n^c} \Pi  \chi_{E_n}   ) = 0.
\end{align}
Then $\widetilde{\Psi}_g\in L^1(\Conf(E), \PP_\Pi)$. If the subspace $\sqrt{g}L$ is closed and the corresponding operator of orthogonal projection $\Pi^g$  is locally of trace class and satisfies, for sufficiently large $R>0$, the condition 
\begin{align}\label{extra}
\tr(\chi_{\{g > R\}} \Pi^g \chi_{\{g>R\}}) <\infty
\end{align}  
then we also have $\PP_{\Pi^g} = \overline{\Psi}_{g}^{\Pi} \cdot \PP_{\Pi}$. 
\end{thm}

\begin{rem}\label{HS-integral}
Note that 
$$
\tr (\chi_{E_n} \Pi | g-1|^2 \chi_{E_n^c} \Pi  \chi_{E_n}   )  = \int_{E_n} d\mu(y) \int_{E_n^c} | g(x)-1|^2 | \Pi(x,y)|^2 d\mu(x).
$$
\end{rem}

Theorem \ref{thm-C} is a strengthening of  and will be derived from \cite[Prop. 4.2]{BQS} which we reformulate here in the form convenient for us.
\begin{prop}[{Proposition 4.2 in \cite{BQS}, particular case}]\label{prop-BQS}
Let $g$ be a nonnegative Borel function on $E$ satisfying $g|_{E_0} =0$, $g|_{E_0^c} >0$ and such that for any $\varepsilon>0$ the subset $A_\varepsilon =  \{x \in E: | g(x) - 1| \ge \varepsilon \} $ is  bounded and 
\begin{align}\label{sing-Buf}
\intt\limits_{A_\varepsilon} | g(x)-1| \Pi(x,x)d\mu(x) < \infty;
\end{align}
\begin{align}\label{decay-Buf}
 \intt\limits_{A^c_\varepsilon} | g(x)-1|^2 \Pi(x,x)d\mu(x)< \infty.
\end{align}
Then $\widetilde{\Psi}_g\in L^1(\Conf(E), \PP_\Pi).$ If the subspace $\sqrt{g}L$ is closed and the corresponding operator of orthogonal projection $\Pi^g$ satisfies, for sufficiently large $R>0$, the condition $\tr(\chi_{\{g > R\}} \Pi^g \chi_{\{g>R\}}) <\infty$,
then we also have $\PP_{\Pi^g} = \overline{\Psi}_{g}^{\Pi} \cdot \PP_{\Pi}$. 
\end{prop} 

\begin{rem} Proposition 4.2 in \cite{BQS} is formulated in slightly greater generality: namely, it still holds if 
$g$ is allowed to take $0$ values on a set $E_0\subset E$ of positive measure , provided that the subset $E_0$ satisfies $\tr(\chi_{E_0} \Pi\chi_{E_0})<\infty$ and that a function $\varphi \in L$ such that $\chi_{E\setminus E_0} \varphi =0$ must be the zero function. This more general formulation is needed in \cite{BQS} in order to cover the case of the discrete phase space when even a finite set of zeros of our function $g$ has positive measure and the requirement states, informally speaking, that no function from $L$ may be supported on a finite set.  
In the continuous case, there is no need for the set $E_0$. At the same time,  Theorem \ref{thm-C} also admits 
a similar  more general version: Theorem \ref{thm-C} still holds if 
$g$ is allowed to take zero values on a set $E_0\subset E$ of positive measure, provided that the subset $E_0$ satisfis $\tr(\chi_{E_0} \Pi\chi_{E_0})<\infty$ and that a function $\varphi \in L$ such that $\chi_{E\setminus E_0} \varphi =0$ must be the zero function.
 \end{rem}

Assumptions of  Theorem \ref{thm-C} are indeed weaker than that of  Proposition \ref{prop-BQS}: under the assumption of Proposition \ref{prop-BQS}, the subsets $E_n  = \{x \in E: | g(x) -1| \ge 1/n\}$ verify all the assumptions of Theorem \ref{thm-C}. Indeed, we have 
\begin{align*}
 \intt\limits_{ E_n^c} | g(x)-1|^3 \Pi(x,x)d\mu(x)  \le \frac{1}{n}  \intt\limits_{ E_n^c} | g(x)-1|^2 \Pi(x,x)d\mu(x)< \infty;
\end{align*}
\begin{align*}
& \iint\limits_{E_n^c \times E_n^c} | g(x) - g(y) |^2 | \Pi(x,y)|^2 d\mu(x)d\mu(y) 
\\
& \le  2 \iint\limits_{E_n^c \times E_n^c} ( | g(x) -1|^2 + |1- g(y) |^2) | \Pi(x,y)|^2 d\mu(x)d\mu(y)   
\\
& \le 4   \int\limits_{ E_n^c} | g(x) -1|^2  \Pi(x,x) d\mu(x)    < \infty,
\end{align*}
while,  by Remark \ref{HS-integral}, 
 \begin{align*}
&\tr (\chi_{E_n} \Pi | g-1|^2 \chi_{E_n^c} \Pi  \chi_{E_n}   )  = \int_{E_n} d\mu(y) \int_{E_n^c} | g(x)-1|^2 | \Pi(x,y)|^2 d\mu(x)
\\
& \le   \int\limits_{ E_n^c} | g(x) -1|^2  \Pi(x,x) d\mu(x) \xrightarrow{n\to\infty} 0.
\end{align*}

\subsection{Outline of the proof of Theorem \ref{thm-C}}

The results in \cite{Buf-multi} and \cite{Buf-inf}
state that if $K\in \mathscr{S}_{1, \loc}(E, \mu)$ defines a determinantal measure $\PP_K$ on $\Conf(E)$ and $g$ is a non-negative bounded measurable function on $E$ such that $\sqrt{|g - 1|} K\sqrt{|g-1|} \in \mathscr{S}_{1}(E, \mu) $ and $1 + (g-1)K$ is invertible, then  the operator 
$$
K^g: = \sqrt{g} K (1 + (g-1) K)^{-1} \sqrt{g} 
$$
induces a determinantal measure $\PP_{K^g}$ on $\Conf(E)$ that coincides with 
$$
\frac{\Psi_g \PP_K}{\displaystyle{\int_{\Conf(E)}} \Psi_g d\PP_K }. 
$$
In other words, a product of a determinantal measure and a multiplicative functional is again 
a determinantal measure given by an explicitly found kernel.
In particular, if $K$ is an orthogonal projection onto a subspace $L\subset L^2(E, \mu)$, then $K^g$ is the orthogonal projection onto the closure of the subspace $\sqrt{g} L$. 

Establishing the equivalence of Palm measures is, however, reduced to proving the equivalence of determinantal point processes $\PP_K$ and $\PP_{K^g}$ when the multiplicative functional $\Psi_g$ is either 
not convergent at all or not integrable with respect to $\PP_K$. We therefore need the formalism of regularized multiplicative functionals in order to establish the desired equivalence. 
 
Proposition 4.2 in \cite{BQS} uses the Hilbert-Carleman regularization of the  Fredholm determinant defined for all
Hilbert-Schmidt operators: $\det_2(1+A)=\det(1+A)\exp(-\tr(A))$. Assumption (\ref{decay-Buf}) precisely ensures that the operator  $\sqrt{|g - 1|} K\sqrt{|g-1|}$ is Hilbert-Schmidt. Unfortunately, this assumption does not hold for reproducing kernels of Hilbert spaces of holomorphic functions, and instead of Hilbert-Schmidt operators we must work with the von Neumann-Schatten class $\mathscr{S}_3$.
Assumption (\ref{decay}) in Theorem \ref{thm-C} ensures the relation  $\sqrt{|g - 1|} K\sqrt{|g-1|}\in\mathscr{S}_3$.
The main step in the proof of Theorem \ref{thm-C} is the extension of the definition of regularized multiplicative functional to this larger class of  functions $g$. The main technical step in the proof is Proposition \ref{technical-prop}.

\section{Case of $\C$}\label{sec-case-fock}
\subsection{Examples}
In this section,  we assume  that $\psi: \C\rightarrow \R$ is a measurable function on $\C$, the condition \eqref{sub-h} is not necessarily satisfied. Recall that we denote $dv_\psi (z) = e^{-2 \psi(z)} d\lambda(z)$ and denote
$\mathscr{F}_\psi  = \left\{f: \C \rightarrow \C\Big| \text{$f$ holomorphic, $\int\limits_\C | f |^2 dv_\psi< \infty$}\right\}.$
If the evaluation functionals $\ev_z (f): =  f(z)$ defined on $\mathscr{F}_\psi$ are uniformly bounded on compact subsets, then $\mathscr{F}_\psi$ is a closed subspace of $L^2(\C, dv_\psi)$. In this case, denote by $B_\psi$ the reproducing kernel of $\mathscr{F}_\psi$, we have 
\begin{align}\label{rep-kernel}
B_\psi(z,w)= \sum_{j=1}^\infty f_j(z) \overline{f_j(w)},
\end{align}
where $(f_j)_{j =1}^\infty$ is any orthonormal basis of $\mathscr{F}_\psi$.

\begin{ass}\label{ass-fock}
The measure $dv_\psi$ satisfies
\begin{itemize}
\item[(1)] the evaluation functionals $\ev_z$ defined on $\mathscr{F}_\psi$ are uniformly bounded on compact subsets; 
\item[(2)] the polynomials are dense in $\mathscr{F}_\psi$;
\item[(3)] $\displaystyle{\int_\C}  \frac{1}{1 + | z|^2}   B_\psi(z,z) dv_\psi(z) < \infty$.
\end{itemize} 
\end{ass}

\begin{exam}[A radial case]\label{ex1}
Let $ \alpha >0$, and set $\psi_\alpha(z) = \frac{1}{2}|z|^{\alpha}$. The measure $dv_{\psi_\alpha} (z) = e^{- |z|^{\alpha}} d\lambda(z)$ satisfies Assumption \ref{ass-fock} if and only if $0 < \alpha < 2$.  Indeed, the first two conditions in Assumption \ref{ass-fock} are satisfied by $dv_{\psi_\alpha}$ by all $\alpha >0$. Now one can see that the third condtion is equivalent to 
\begin{align}\label{norm-series}
\sum_{n =1}^\infty \frac{\| z^{n-1}\|_{L^2(dv_\psi)}^2}{\| z^n\|_{L^2(dv_\psi)}^2} <\infty.
\end{align}
A direct computation shows that
\begin{align}\label{monomials-norm}
 \| z^n\|_{L^2(dv_\psi)}^2 = \frac{2\pi}{\alpha} \Gamma\left(\frac{2n+2}{\alpha}\right) \et  \frac{\| z^{n-1}\|_{L^2(dv_\psi)}^2}{\| z^n\|_{L^2(v_\psi)}^2} \sim  \frac{1}{n^{2/\alpha}}.
\end{align}
The series \eqref{norm-series} converges if and only if  $0< \alpha <2$.
\end{exam}

\begin{rem}\label{rem-reason}
As shown in Example \ref{ex1},  the third condition in Assumption \ref{ass-fock} is too strict: indeed, it fails already for the Ginibre point process (corresponding to $\psi(z)  = \frac{1}{2}|z|^2$). 
\end{rem}

\medskip

Let $\PP_{B_\psi}$ be the determinantal point process induced by the operator $B_\psi$. For any $\ell$-tuple $\mathfrak{q} = (q_1, \dots, q_\ell) \in \C^\ell$ of distinct points, set 
$$
\mathscr{F}_\psi(\mathfrak{q})  : =\left \{ f \in \mathscr{F}_\psi\Big| f (q_1) = \dots = f(q_\ell)  =0 \right\},
$$
and let $B_\psi^{\mathfrak{q}}$ denote the operator of orthogonal projection onto $\mathscr{F}_\psi (\mathfrak{q})$.  Recall that the Palm distribution $\PP_{B_\psi}^{\mathfrak{q}}$ of $\PP_{B_\psi}$ conditioned at $\mathfrak{q}$ is induced by $B_\psi^{\mathfrak{q}}$, i.e., 
$$
\PP_{B_\psi}^{\mathfrak{q}} = \PP_{B_\psi^{\mathfrak{q}}}.
$$

Given a positive integer $\ell \in \N$,  introduce the closed subspace 
\begin{align}\label{symmetric}
\mathscr{F}_\psi^{(\ell)} : = \left\{ f \in \mathscr{F}_\psi \Big | f(0) = f'(0) = \dots = f^{(\ell- 1)} (0) = 0 \right\}.
\end{align}
Denote  $B_\psi^{(\ell)}$ the operator of orthogonal projection onto $\mathscr{F}_\psi^{(\ell)}$. Let $\PP_{B_\psi}^{(\ell)}$ be the determinantal point process induced by $B_\psi^{(\ell)}$.

\begin{rem}
In general, we do not have $\mathscr{F}_\psi^{(\ell)} = z^\ell \mathscr{F}_\psi.$ Indeed, let $\psi(z) = \half |z|^2$,  we have $z \mathscr{F}_\psi \not\subset \mathscr{F}_\psi$. This can be seen from the closed graph theorem: otherwise, the operator $M_z: \mathscr{F}_\psi \rightarrow \mathscr{F}_\psi$ of multiplication by the function $z$ is bounded, which contradicts the explicit computation \eqref{monomials-norm}:
$$
\n M_z\n_{\mathscr{F}_\psi \rightarrow \mathscr{F}_\psi} \ge \sup_{n} \frac{\n z^{n+1}\n_{\mathscr{F}_\psi}}{\n z^n\n_{\mathscr{F}_\psi}} = \infty;
$$
see also the related discussion after Theorem 2 in \cite{Ferguson}.
\end{rem}

\begin{prop}\label{prop-fock}
If $\psi$ satisfies Assumption \ref{ass-fock}, then for any $\ell\in \N$  and  any $\ell$-tuple $\mathfrak{q} = (q_1, \dots, q_\ell) \in \C^\ell$ of distinct points,  we have equivalence of measures: 
$$
\PP_{B_\psi}^{\mathfrak{q}} \simeq \PP_{B_\psi}^{(\ell)}.
$$
Moreover, if one sets 
$$
g_{\mathfrak{q}}(z) =  \left|   \frac{(z-q_1)\dots (z - q_\ell)}{z^\ell}  \right|^2, 
$$
then the Radon-Nikodym derivative is given by the regularized multiplicative functional 
$$
\frac{d\PP_{B_\psi}^{\mathfrak{q}}}{d \PP_{B_\psi}^{(\ell)}} = \overline{\Psi}_{g_\mathfrak{q}}^{B_\psi^{(\ell)}}.
$$
In particular, given any two $\ell$-tuples $\mathfrak{q}$ and $\mathfrak{q}'$ of distinct points, the corresponding Palm measures $\PP_{B_\psi}^{\mathfrak{q}}$ and $\PP_{B_\psi}^{\mathfrak{q}'}$ are equivalent.
\end{prop}

\begin{proof}
First note that, under Assumption \ref{ass-fock}, for any $\ell\in \N$  and  any $\ell$-tuple $\mathfrak{q} = (q_1, \dots, q_\ell) \in \C^\ell$ of distinct points, 
$$
\mathscr{F}_\psi (\mathfrak{q}) = \frac{(z-q_1)\dots (z - q_\ell)}{z^\ell} \mathscr{F}_\psi^{(\ell)}.
$$
Indeed, if $f\in \mathscr{F}_\psi^{(\ell)}$, then the function $h(z): =  \frac{(z-q_1)\dots (z - q_\ell)}{z^\ell} f(z)$ is holomorphic on $\C$ and vanishes at $q_1, \cdots, q_\ell$. Moreover, 
\begin{align*}
 \int\limits_\C | h  |^2 dv_\psi 
& =  \int_{\D} | h |^2 dv_\psi   +  \int_{\C\setminus \D} |  h  |^2 dv_\psi 
 \\
 &\le  v_\psi (\D) \cdot  \sup_{z\in \D} | h(z) |^2 + \sup_{z\in \C\setminus \D}  \left| \frac{(z-q_1)\dots (z - q_\ell)}{z^\ell}\right|^2 \int_{\C\setminus \D} |  f |^2 dv_\psi    < \infty.
\end{align*}
Hence we get $h \in  \mathscr{F}_\psi (\mathfrak{q})$.  Conversely, if $h \in  \mathscr{F}_\psi (\mathfrak{q})$, then similar proof as above shows that $f(z) :  = \frac{z^\ell} {(z-q_1)\dots (z - q_\ell)}h(z) $ is a function in  $\mathscr{F}_\psi^{(\ell)}$.

By the elementary fact from Remark \ref{rem-unit}, the operator of orthogonal projections from $L^2(\C, v_\psi)$ onto the following two subspaces 
\begin{align*}
 \frac{(z-q_1)\dots (z - q_\ell)}{z^\ell} \mathscr{F}_\psi^{(\ell)}    \text{ and }  \left| \frac{(z-q_1)\dots (z - q_\ell)}{z^\ell}\right| \mathscr{F}_\psi^{(\ell)}  = \sqrt{g_{\mathfrak{q}}} \cdot \mathscr{F}_\psi^{(\ell)}
\end{align*}
induce the same determinantal point process. Consequently,  for finishing the proof of Proposition \ref{prop-fock}, it suffices to  verify that the pair $(B_\psi^{(\ell)}, g_{\mathfrak{q}})$ satisfies all the assumptions of Proposition \ref{prop-BQS}. Note that by representing $B_\psi^{(\ell)}$ in similar  form as \eqref{rep-kernel}, we have $ B_\psi^{(\ell)}(z,z)  = O(  |z|^{2\ell})$ for $|z| \to 0$. Hence there exists $C> 0$, such that
\begin{align*}
& \intt_{  | z| \le 1 } | g_{\mathfrak{q}}(z) - 1| B_\psi^{(\ell)}(z,z) dv_\psi(z)    \le  C \cdot \sup_{| z| \le 1}  ( | g_{\mathfrak{q}}(z) - 1|\cdot |z|^{2\ell} ) \cdot  \intt_{  | z| \le 1 } dv_\psi(z)   < \infty. 
\end{align*}
On the other hand, $ |g_{\mathfrak{q}}(z) -1 |^2  =  O\left( 1/|z|^2  \right)$ as $| z | \to \infty$. Recalling that $B_\psi^{(\ell)}(z,z) \le B_\psi(z,z)$, we have
\begin{align*}
 \intt_{ | z | \ge 1 } | g_{\mathfrak{q}}(z) - 1|^2B_\psi^{(\ell)}(z,z) dv_\psi(z) \le  \sup_{| z|\ge 1} (|z|^2 | g_{\mathfrak{q}}(z)- 1|^2 )    \cdot \intt_{ | z | \ge 1 } \frac{1}{| z|^2}B_\psi (z,z) dv_\psi(z) .
\end{align*}
By the third condition in Assumption \ref{ass-fock}, we may conclude that  the above integral is finite. Since $\mathscr{F}_\psi (\mathfrak{q})$ is a closed subspace in $L^2(\C, v_\psi)$, so is  $\sqrt{g_{\mathfrak{q}}} \cdot \mathscr{F}_\psi^{(\ell)}$.  Moreover,  there exists a function $\alpha: \C \rightarrow \C$ such that $|\alpha(z) | =1$ and the orthogonal projection from $L^2(\C, v_\psi)$ onto the subspace $\sqrt{g_{\mathfrak{q}}} \cdot \mathscr{F}_\psi^{(\ell)}$  is given by
\begin{align*}
 [B_\psi^{(\ell)}]^{g_{\mathfrak{q}}} = \alpha \cdot B_\psi^{\mathfrak{q}} \cdot \overline{\alpha}. 
 \end{align*} 
 It follows that, for sufficiently large $R>0$, since the set $\{z \in \C: g_{\mathfrak{q}} (z) > R\}$ is bounded, we have 
 \begin{align*}
\tr(\chi_{\{ g_{\mathfrak{q}}  > R\} }  [B_\psi^{(\ell)}]^{g_{\mathfrak{q}}}   \chi_{\{ g_{\mathfrak{q}} >R\} })     = \tr(\chi_{ \{ g_{\mathfrak{q}}  > R\} }   \alpha \cdot B_\psi^{\mathfrak{q}} \cdot \overline{\alpha}   \chi_{\{ g_{\mathfrak{q}} >R \} })        = \int_{\{g_{\mathfrak{q}}  > R \}}   B_\psi^{\mathfrak{q}}  (z, z) d v_\psi(z)< \infty. 
 \end{align*}

The proof of Proposition \ref{prop-fock} is complete. 
\end{proof}

\subsection{Proof of Theorem \ref{thm-A2}}\label{derivation-fock}
We now derive Theorem \ref{thm-A2} from Theorem \ref{thm-C}. From now on, the function $\psi$ is assumed to satisfy the condition \eqref{sub-h} until the end of this paper. 

Let $\ell\ge 1$  and let $\mathfrak{p} = (p_1, \dots, p_\ell) , \mathfrak{q} = (q_1, \dots, q_\ell) \in \C^\ell$ be any  two fixed $\ell$-tuples of distinct points; let $g$ be the function defined by the formula
\begin{align}\label{def-m-g}
g(z) = |g_{\mathfrak{p}, \mathfrak{q}}(z)|^2 = \left| \frac{(z-p_1) \cdots (z-p_\ell)}{ (z - q_1) \cdots (z-q_\ell)} \right|^2.
\end{align}
 Let $0< \varepsilon<1$ be a small fixed number.  Choose 
 $$R_\varepsilon > \max\{ |p_k|, | q_k|: k = 1, \dots, \ell\}$$ large enough,  such that outside the following subset $$A_\varepsilon= \{z \in \C: | z | \le R_\varepsilon \},$$ we have $| g(z) - 1| \le \varepsilon$. Finally, for $n\in \N$, let $$E_n = \{ z\in \C: | z| \le \max(R_\varepsilon, n)\}.$$

We start with a simple but very useful observation that conditions  \eqref{decay}, \eqref{variance}, \eqref{flat-cut} and \eqref{extra}  in Theorem \ref{thm-C} are preserved under taking finite rank pertubation.

\begin{rem}\label{p-rem}
Assume that the pair $(g, \Pi)$ satisfies the conditions  \eqref{decay}, \eqref{variance}, \eqref{flat-cut} and \eqref{extra}  in Theorem \ref{thm-C}.    If $\widetilde{\Pi}  = \Pi + \Pi'$, where $\Pi'$ has finite rank and $\Ran(\Pi) \perp \Ran(\Pi')$, or  $ \widetilde{\Pi} = \Pi - \Pi'$, where $\Pi'$ has finite rank and $\Ran(\Pi') \subset \Ran(\Pi)$, then conditions  \eqref{decay}, \eqref{variance}, \eqref{flat-cut} and \eqref{extra} hold for the new pair $(g, \widetilde{\Pi})$ . If $g$ is unbounded, then the condition \eqref{singularity} for the pair $(g, \Pi)$ does not imply the condition for the pair $(g, \widetilde{\Pi})$. The condition \eqref{singularity} is on the other hand usually easy to check directly.
\end{rem}

\begin{lem}\label{lem-signularity-decay}
Let $g$ be the function defined by the formula \eqref{def-m-g} and let $E_n$. We have 
$$ 
\intt\limits_{ E_n} | g(z)-1| B_\psi^{\mathfrak{q}}(z,z) e^{-2\psi(z)}d\lambda(z)< \infty;
 $$
 $$
 \intt\limits_{ E_n^c} | g(z)-1|^3 B_\psi^{\mathfrak{q}}(z,z)e^{-2\psi(z)}d\lambda(z)< \infty.
$$
\end{lem}
\begin{proof}
We first note that for any $n\in\N$,  here exists $C_n  > 0$ such that 
\begin{align}\label{around-q}
  B_\psi^{\mathfrak{q}}(z,z) \le  C_n \cdot \prod_{k=1}^\ell |z-q_k|^2,  \text{ for any $z\in E_n$}. 
\end{align}
Indeed,  $B_\psi^{\mathfrak{q}}$ is  the reproducing kernel of the subspace $\mathscr{F}_\psi(\mathfrak{q})$, holomorphic in the first coordinate and anti-holomorphic in  the second coordinate.  Hence for any $w\in\C$, the function $z \mapsto B_\psi^{\mathfrak{q}}(z,w)$ belongs to $\mathscr{F}_\psi(\mathfrak{q})$, that is, it is holomorphic and vanishes at $q_1, \cdots, q_\ell$. Consequently, we may write 
\begin{align*}
B_\psi^{\mathfrak{q}}(z, w) = \prod_{k=1}^\ell (z-q_k) \cdot h(z,w). 
\end{align*} 
where $h(z,w)$ is  holomorphic in the first coordinate and anti-holomorphic in  the second coordinate. Since $B_\psi^{\mathfrak{q}}(z,w)$ is a Hermitian kernel, we may write further 
\begin{align}\label{kernel-palm}
B_\psi^{\mathfrak{q}}(z, w) = \prod_{k=1}^\ell (z-q_k) (\bar{w} -\bar{q_k}) \cdot t(z,w),
\end{align} 
where $t(z,w)$ is a continuous function,   holomorphic in the first coordinate and anti-holomorphic in  the second coordinate.  Taking $C_n = \sup_{z\in E_n} t(z,z)$, we get the desired inequality \eqref{around-q}.  Now we have
\begin{align*}
&\sup_{z\in E_n} |g(z)-1| \cdot B_\psi^{\mathfrak{q}}(z, z) 
\\
\le& C_n \sup_{z\in E_n}  \Big|| (z-p_1) \cdots (z-p_\ell) |^2 -| (z - q_1) \cdots (z-q_\ell)|^2\Big| <\infty,
\end{align*}
and the first inequality in the lemma follows immediately.

By Theorem \ref{Christ}, there exists a constant $C>0$, such that 
$$
B_\psi^{\mathfrak{q}}(z,z)e^{-2\psi(z)} \le B_\psi(z,z) e^{-2\psi(z)} \le C.
$$
Since $| g(z) - 1|^3 = O (1/| z|^3)\text{\,\,as\,\,} |z| \to \infty$, there exists $C'>0$, such that 
$$
\intt\limits_{ E_n^c} | g(z)-1|^3 B_\psi^{\mathfrak{q}}(z,z)e^{-2\psi(z)}d\lambda(z) \le C'  \int_{| z| \ge R_\varepsilon} \frac{1}{| z|^3} d\lambda(z)<\infty.
$$
\end{proof}

\begin{rem}\label{rem-alt}
An alternative proof of the inequality \eqref{around-q} is as follows. From the theory of reproducing kernel Hilbert spaces, we have 
\begin{align}\label{cal-diag}
 B_\psi^{\mathfrak{q}}(z,z)  = \sup_{f \in \mathscr{F}_\psi(\mathfrak{q}), \| f\|_{L^2(\C, v_\psi)} = 1} |f(z)|^2. 
\end{align}
By Closed Graph Theorem, the map $f \mapsto \frac{f(z) }{\prod_{k=1}^\ell (z  - q_k)}$ induces a bounded linear operator from $\mathscr{F}_\psi(\mathfrak{q})$ to $C(E_n)$, where $C(E_n)$ is the space of  continuous functions on the compact set $E_n$. Consequently, by denoting $C_n$ the operator norm of the above bounded linear operator, that is, 
\begin{align*}
C_n = \sup_{f \in \mathscr{F}_\psi(\mathfrak{q}), \| f\|_{L^2(\C, v_\psi)} = 1} \sup_{z\in E_n}  \left| \frac{f(z) }{\prod_{k=1}^\ell (z  - q_k)}\right|^2<\infty,
\end{align*}
and using \eqref{cal-diag}, we get the desired inequality \eqref{around-q}.
\end{rem}

\begin{lem}\label{lem-variance}
Let $g$ be the function defined by the formula \eqref{def-m-g}. We have 
 \begin{align}\label{ginibre-variance}
 \iint\limits_{A_\varepsilon^c \times A_\varepsilon^c} | g(z) - g(w) |^2 | B_\psi^{\mathfrak{q}}(z,w)|^2 dv_\psi(z)dv_\psi(w) < \infty.
 \end{align}
\end{lem}

\begin{proof}
Since $B_\psi^{\mathfrak{q}}$ is a finite rank perturbation of  $B_\psi$, and since $g$ is bounded on $A_\varepsilon^c$, it suffices to show  that
\begin{align}\label{I1}
I_1: =\iint_{| z| \ge R_\varepsilon, | w | \ge R_\varepsilon }  | g (z)-  g (w)|^2 | B_\psi(z,w)|^2  dv_\psi(z) dv_\psi(w) < \infty.
\end{align}
By  the definition \eqref{def-m-g} of  $g$, there exist $\alpha_1, \cdots, \alpha_\ell  \in\C$, such that 
\begin{align*}
g(z) = \left|   1  + \sum_{k=1}^\ell \frac{\alpha_k}{z-q_k} \right|^2.
\end{align*}
Hence for any $z, w\in A_\varepsilon^c$, we have
\begin{align*}
| g (z)-  g (w)|    \le \sup_{z, w\in A_\varepsilon^c}   \left(\Big|   1  + \sum_{k=1}^\ell \frac{\alpha_k}{z-q_k} \Big| + \Big|   1  + \sum_{k=1}^\ell \frac{\alpha_k}{w-q_k} \Big|\right)  \left|   \sum_{k=1}^\ell \frac{\alpha_k}{z-q_k}  -   \frac{\alpha_k}{w-q_k}   \right|.
\end{align*}
Note also that 
\begin{align*}
 \sup_{z, w\in A_\varepsilon^c}   \frac{ \left|   \frac{1}{z-q_k}  -   \frac{1}{w-q_k}   \right|}{ \left|  \frac{1}{z}  -   \frac{1}{w}   \right|} < \infty. 
\end{align*}
It follows that there exists $C_\varepsilon>0$, such that for any $z, w\in\C$ with $|z| \ge R_\varepsilon, |w| \ge R_\varepsilon$,  we have
\begin{align*}
| g (z)-  g (w)| \le C_\varepsilon \left| \frac{1}{z} - \frac{1}{w}\right|. 
\end{align*}
Now Christ's pointwise estimate, \eqref{pt-es} in Theorem \ref{Christ}, implies that for proving \eqref{I1}, it suffices to prove  
\begin{align}\label{I2}
I_2: = \iint_{| z| \ge R_\varepsilon, | w | \ge R_\varepsilon }  \left| \frac{1}{z} - \frac{1}{w}\right|^2  e^{- \delta | z-w|} d\lambda(z) d\lambda(w)< \infty.
\end{align}
To this end, we write
\begin{align*}
I_2   = &  \int_{| w | \ge R_\varepsilon }    d\lambda(w) \int_{| \zeta +w| \ge R_\varepsilon}  \frac{| \zeta|^2}{| w (w + \zeta)|^2 } e^{- \delta | \zeta|} d\lambda(\zeta) 
\\  
 =& \int_{| w | \ge R_\varepsilon }    d\lambda(w)   \int_{| \zeta +w| \ge R_\varepsilon}     \chi_{\{| w | \ge 2 | \zeta|\}} \frac{| \zeta|^2}{| w (w + \zeta)|^2 } e^{- \delta | \zeta|} d\lambda(\zeta) 
\\ 
& +  \int_{| w | \ge R_\varepsilon }    d\lambda(w)   \int_{| \zeta +w| \ge R_\varepsilon}     \chi_{\{| w | < 2 | \zeta|\}} \frac{| \zeta|^2}{| w (w + \zeta)|^2 } e^{-\delta | \zeta|} d\lambda(\zeta).
\end{align*}
The first integral is controlled by 
\begin{align*}
4  \int_{| w | \ge R_\varepsilon }    d\lambda(w)   \int_{\C}    \frac{| \zeta|^2}{| w|^4 } e^{- \delta | \zeta|} d\lambda(\zeta) < \infty,
\end{align*}
while the second integral is controlled by
\begin{align*}
& \int_{| w | \ge R_\varepsilon }    d\lambda(w)   \int_{\C}     \chi_{\{| w | < 2 | \zeta|\}} \frac{| \zeta|^2}{| R_\varepsilon w  |^2 } e^{-\delta | \zeta|} d\lambda(\zeta)\\ = & 2 \pi \int_{ 2 | \zeta| \ge R_\varepsilon} \log \left( \frac{2| \zeta|}{R_\varepsilon}\right) \frac{| \zeta|^2}{| R_\varepsilon   |^2 } e^{-\delta | \zeta|} d\lambda(\zeta)<\infty.
\end{align*} 
The proof of the lemma is complete.  
\end{proof}

\begin{lem}\label{lem-flatcut}
Let $g$ be the function defined by the formula \eqref{def-m-g}. We have 
\begin{align}\label{ginibre-flatcut}
\lim_{n\to \infty}\tr (\chi_{E_n} B_\psi^{\mathfrak{q}} | g-1|^2 \chi_{E_n^c} B_\psi^{\mathfrak{q}}  \chi_{E_n}   ) = 0.
\end{align}
\end{lem}

\begin{proof}
Since $B_\psi^{\mathfrak{q}}$ is a finite rank perturbation of $B_\psi$,  by Remark \ref{p-rem}, it suffices to check the same condition \eqref{ginibre-flatcut} for  the new pair $(g, B_\psi)$.
Applying again Christ's pointwise estimate \eqref{pt-es}, we have
\begin{align*}
I_3(n): = & \tr (\chi_{E_n} B_\psi | g-1|^2 \chi_{E_n^c} B_\psi  \chi_{E_n}   )  
 =  \| \chi_{E_n} B_\psi | g-1|\chi_{E_n^c} \|_{HS}^2 
 \\ 
 =  &  \intt_{| z| \le n}  \intt_{|w| \ge n}  | g(w) - 1|^2 |B_\psi(z,w)|^2  e^{-2\psi(z)-2 \psi(w)} d\lambda(z) d\lambda(w) 
 \\ 
 \le & C  \intt_{| z| \le n} \intt_{|w| \ge n} | g(w)-1|^2 e^{- \delta | z-w|}d\lambda(z)d\lambda(w) 
 \\ 
 \le & C'  \intt_{| z| \le n} \intt_{|w| \ge n} \frac{1}{|w|^2} e^{-\delta |z-w|} d\lambda(z)d\lambda(w)  = C'  \intt_{|w| \ge n} \frac{d\lambda(w)}{|w|^2}\intt_{| w + \zeta| \le n } e^{- \delta | \zeta|}d \lambda( \zeta) \\ 
  \le & C' \intt_{| w| \ge n}\frac{d\lambda(w)}{|w|^2} \intt_{  | w| - n \le | \zeta| \le | w|  + n  } e^{- \delta | \zeta|} d \lambda(\zeta)
  = 4\pi^2C' \int_{s\ge n} \frac{ds}{s} \int_{s-n}^{s+n} r e^{- \delta r } dr.
  \end{align*}
  Now since there exists $C''>0$, such that $r e^{- \delta r} \le C''e^{- \delta r /2}$ for all $r \ge 0$, we have
  \begin{align*}
  I_3(n) \le  C''' \intt_{s \ge n}\frac{e^{-\delta (s-n)/2}  }{s} ds  = C''' \int_{1}^\infty \frac{e^{-\delta n(t-1)/2}  }{t}dt.
\end{align*}
By dominated convergence theorem, we have $\lim_{n\to \infty} I_3(n) = 0$. 
\end{proof}

\begin{proof}[Proof of Theorem \ref{thm-A2}]
By Lemma \ref{lem-signularity-decay}, Lemma \ref{lem-variance} and Lemma \ref{lem-flatcut}, the conditions \eqref{singularity}, \eqref{decay}, \eqref{variance}, \eqref{flat-cut} are satisfied by the pair $(g, B_\psi^{\mathfrak{q}})$. Moreover, let 
$$
\alpha(z) = \frac{|g_{\mathfrak{p},\mathfrak{q}} (z)| }{g_{\mathfrak{p},\mathfrak{q}} (z) },
$$
then by Proposition \ref{prop-F-rel}, we have 
$$
\sqrt{g(z)} \mathscr{F}_\psi(\mathfrak{q}) = \alpha (z) g_{\mathfrak{p},\mathfrak{q}} (z)  \mathscr{F}_\psi(\mathfrak{q}) = \alpha(z) \mathscr{F}_\psi(\mathfrak{p}).
$$
Hence $\sqrt{g(z)} \mathscr{F}_\psi(\mathfrak{q})$ is a closed subspace of $L^2(dv_\psi)$. And   $(B_\psi^{\mathfrak{q}})^{g} =   \alpha B_\psi^{\mathfrak{p}} \bar{\alpha}$ is locally of trace class, this implies the condition \eqref{extra}. Now the formula  \eqref{p-q-rn}  of  Radon-Nikodym derivative $d\PP_{B_\psi}^{\mathfrak{p}}/d\PP_{B_\psi}^{\mathfrak{q}}$ follows from Theorem \ref{thm-C} and Remark \ref{rem-unit}. 
\end{proof}

\begin{rem}
Under the condition \eqref{sub-h}, we also have the same result as in Proposition \ref{prop-fock}.
\end{rem}

\subsection{Proof of Proposition \ref{thm-rigid}}

\begin{lem}\label{lem-d-norm}
There exists a constant $C>0$, depending only on $\psi$, such that for any $C^2$-smooth compactly supported function $\varphi: \C \rightarrow \R$, we have
\begin{align}\label{var-by-d}
\Var_{\PP_{B_\psi}} (S_\varphi)\le C \int_\C \| \nabla \varphi (w)\|_2^2 d\lambda(w).
\end{align}
\end{lem}

\begin{proof}

Let $\varphi: \C \rightarrow \R$ be a $C^2$-smooth compactly supported function. Our convention for the Fourier transform of $\varphi$ will be 
$$
\widehat{\varphi} (\xi) = \int_\C \varphi(w)  e^{- i 2\pi \langle w, \xi \rangle} d\lambda(w), \text{ where }\langle z, w \rangle : =  \Re (z) \Re (w) + \Im (z) \Im(w).
$$

By definition, we have 
$$
\Var_{\PP_{B_\psi}} (S_\varphi) = \frac{1}{2}\iint_{\C^2} | \varphi(z) - \varphi(w)|^2 | B_\psi(z,w)|^2 e^{-2\psi(z)- 2 \psi(w)} d\lambda(z) d\lambda(w).
$$
By Theorem \ref{Christ} and Plancherel identity for Fourier transform, we  obtain 
\begin{align*}
\Var_{\PP_{B_\psi}} (S_\varphi) & \le C \iint_{\C^2} | \varphi(z) - \varphi(w)|^2  e^{- \delta | z  - w|} d\lambda(z) d\lambda(w) 
\\
& =  C \iint_{\C^2} | \varphi(\zeta + w) - \varphi(w)|^2  e^{- \delta | \zeta|}  d\lambda(w) d\lambda(\zeta)
\\ 
& = C \iint_{\C^2} | e^{i 2 \pi \langle \xi , \zeta \rangle } - 1 |^2 | \widehat{\varphi} (\xi)|^2 e^{- \delta | \zeta|}    d\lambda(\xi) d\lambda(\zeta).
\end{align*}
Now since $| e^{i 2 \pi \langle \xi, \zeta \rangle } - 1|  = 2 |\sin ( \pi \langle \xi, \zeta \rangle ) |  \le 2 \pi | \xi| | \zeta|$, we have 
\begin{align*}
\Var_{\PP_{B_\psi}} (S_\varphi) 
&\le C'   \iint_{\C^2}  | \xi|^2 | \widehat{\varphi} (\xi)|^2  | \zeta|^2e^{- \delta | \zeta|}    d\lambda(\xi) d\lambda(\zeta)
\\ & \le C''\int_{\C}  | \xi|^2 | \widehat{\varphi} (\xi)|^2   d\lambda(\xi) = C'' \int_{\C} \| \nabla \varphi (w)\|_2^2 d\lambda(w).
\end{align*}
\end{proof}

\begin{proof}[Proof of Proposition \ref{thm-rigid}]
We will follow the argument of Ghosh and Peres \cite{Ghosh-rigid}. By Theorem \ref{thm-GP}, it suffices, for any fixed $R>0$ and any $\varepsilon > 0$, to construct a function  $\Phi_{\varepsilon, R} \in C_c^2(\C)$ such that  $\Phi_{\varepsilon, R} (z)  \equiv 1$ whenever $| z| \le R$ and $\Var_{\PP_{B_\psi}} (S_{\Phi_{\varepsilon, R}}) < \varepsilon$. 

Let $r_0 = 2 R$.  By Lemma \ref{lem-d-norm}, it suffices to construct a radial function $$\Phi_{\varepsilon, R}(z) = \phi_{\varepsilon, R}(|z|),$$ 
with $\phi_{\varepsilon, R}$ a function in $C_c^2(\R_{+})$ such that $\phi_{\varepsilon, R}|_{[0, r_0/2]} \equiv 1$ and 
$$
\int_{0}^\infty | \phi_{\varepsilon, R}'(r)|^2 rdr < \varepsilon.
$$ 
To this end, first we take $\tilde{\phi}_{\varepsilon, R}(r) = (1- \varepsilon \log^{+}(r/r_0))_{+}$, where $\log^{+} (x)=  \max (\log x, 0)$. Note that $\tilde{\phi}_{\varepsilon, R} |_{[r_0\exp(1/\varepsilon), \infty)} \equiv 0$ and $\tilde{\phi}_{\varepsilon, R}'(r) = - \varepsilon/r$ on the interval $(r_0, r_0 \exp(1/\varepsilon))$. Next we smooth the function $\tilde{\phi}_{\varepsilon, R}$ at the points $r_0$ and  $r_0\exp(1/\varepsilon)$ to obtain a function $\phi_{\varepsilon, R} \in C_c^2(\R_{+})$ such that  $\phi_{\varepsilon, R}$ identically equals to $1$ on $[0, r_0/2]$ and $\phi_{\varepsilon, R}'$ is supported on $[r_0/2, 2r_0\exp(1/\varepsilon)]$ such that $|\phi_{\varepsilon, R}'(r)| \le \varepsilon /r$ for all $r>0$. Hence we have 
$$
\int_{0}^\infty | \phi_{\varepsilon, R}'(r)|^2 rdr  \le \int_{r_0/2}^{2r_0\exp(1/\varepsilon)} \frac{\varepsilon^2}{r}dr = \varepsilon + \varepsilon^2 \log 4.
$$
This completes the proof of the proposition.
\end{proof}

\section{Case of $\D$}\label{sec-case-bergman}

\subsection{Analysis of the conditions on the weight $\omega$}\label{sec-w-disk}
Let $\omega: \D \rightarrow \R^{+}$ be a Bergman weight. We collect some known results from the literature on the sufficient conditions on the Bergman weight $\omega$, so that the inequality \eqref{w-condition}:
$$
\int_\D (1 - | z|)^2 B_\omega(z,z) \omega(z)d\lambda(z)< \infty
$$
 holds.

\begin{exam}[Classical weights]
Assume $\omega(z) = ( 1 - | z|^2)^{\alpha}, \alpha > -1$. Then
$$
B_\omega(z,w) =  \frac{\alpha+1}{\pi} \frac{1}{(1 - z \bar{w})^{\alpha + 2}},
$$
hence $(1 - | z|)^2 B_\omega(z,z) \omega(z)$ is bounded and the inequality \eqref{w-condition} holds.
\end{exam}

\begin{exam}[A class of logarithmatically superharmonic weights]
Let $\omega(z) = e^{-2\varphi(z)}.$
Assume  that 
\begin{itemize}
\item[1)]  $\varphi \in C^2(\D)$ and $\Delta \varphi > 0$;
\item[2)] the function $(\Delta \varphi (z))^{-1/2}$ is Lipschitz on $\D$; 
\item[3)] there exist  $C_1, a > 0$ and $0 < t < 1$, such that  
$$
(\Delta \varphi (z))^{-1/2} \le C_1 ( 1 - | z|);
$$  
$$
(\Delta \varphi (z))^{-1/2} \le (\Delta \varphi (w))^{-1/2} + t | z - w| \text{\, for \,} | z- w| > a (\Delta \varphi (w))^{-1/2}.
$$
\end{itemize}
By  \cite[Lemma 3.5]{Lin-Rochberg}, the weight $\omega$ is a Bergman weight and 
$$
\sup_{z \in \D}(1 - |z|)^2B_\omega(z,z) \omega(z) < \infty.
$$ 
Hence the inequality \eqref{w-condition} holds. Some concrete such examples are 
\begin{itemize}
\item $\omega(z) = (1- | z|^2)^\alpha \exp(h(z))$ with $ \alpha > 0$ and $h(z)$ any real harmonic function on $\D$;
\item $\omega(z) = (1 - | z|^2)^\alpha \exp(-\beta(1-| z|^2)^{-\gamma} + h(z))$ with $\alpha \ge 0, \beta>0, \gamma >0$ and $h(z)$ any real harmonic function on $\D$. 
\end{itemize}
\end{exam}

\begin{prop}
Let $\omega_1, \omega_2$ be two Bergman weights on $\D$ such that 
$$
\int_\D (1 - | z|)^2 B_{\omega_1} (z,z) \omega_2(z)d\lambda(z) < \infty.
$$
Let  $\omega$ be a Bergman weight on $\D$ and assume that there exist $c, C > 0$ such that 
$$
c \omega_1(z) \le \omega(z) \le C \omega_2(z).
$$
Then $\omega$ satisfies the condition \eqref{w-condition}.
\end{prop}

\begin{proof}
Since $B_\omega(z,z)= \sup_{\n f\n_{\mathscr{B}_\omega \le 1}} | f(z)|^2,$
we have $B_{\omega}(z,z) \le c^2 B_{\omega_1}(z,z)$. By the assumption, we have
$$
\int_\D (1 - | z|)^2 B_{\omega} (z,z) \omega(z)d\lambda(z) \le c^2 C \int_\D (1 - | z|)^2 B_{\omega_1} (z,z) \omega_2(z)d\lambda(z) < \infty.
$$
\end{proof}

\begin{exam}
Let $\omega$ be a Bergman weight. Assume that there exist $c, C>0$ and  $\alpha, \beta$ satisfying either the inequality $ 0\ge \alpha \ge \beta > -1$ or  the inequality $\alpha \ge \beta > \alpha - 1 \ge -1$ and  such that 
$$
c (1 - |z|^2)^\alpha \le \omega(z) \le C (1-|z|^2)^\beta.
$$
Then $\omega$ satisfies the condition \eqref{w-condition}.
\end{exam}

\subsection{Proof of Theorem \ref{thm-B2} and Proposition \ref{prop-B}}\label{derivation-berg}
Let $k, \ell \in \N\cup \{0\}$, let  $\mathfrak{p} \in \D^\ell $ be an $\ell$-tuple of {\it distinct} points and $\mathfrak{q}\in \D^k$ a $k$-tuple of distinct points. Set 
$$
g(z) = | b_{\mathfrak{p}}(z) b_{\mathfrak{q}}(z)^{-1}|^2 = \prod_{j= 1}^\ell \left| \frac{z-p_j}{1 - \bar{p}_j z } \right|^2 \cdot \prod_{j=1}^k\left| \frac{1- \bar{q}_j z }{z-q_j} \right|^2.
$$

By virtue of Proposition \ref{prop-B-rel}, to prove Proposition \ref{prop-B} and hence Theorem \ref{thm-B2}, it suffices to prove that the pair $(g, B_\omega^{\mathfrak{q}})$ satisfies the assumption of Proposition \ref{prop-BQS}. This is done in the following

\begin{lem}
Take $\varepsilon > 0$ small enough and let  $D_\varepsilon = \bigcup_{i = 1}^k U_\varepsilon(q_i) $, where $U_\varepsilon (q_i)$ is a disc centred at point $q_i$ with radius $\varepsilon$  in $\D$. Then we have 
\begin{align}\label{two-integrals}
\int_{D_\varepsilon} | g(z)-1| B_\omega^{\mathfrak{q}}(z,z) \omega(z) d\lambda(z)+ \int_{D_\varepsilon^c} | g(z)-1|^2 B_\omega^{\mathfrak{q}}(z,z) \omega(z)d\lambda(z) <\infty.
\end{align}
\end{lem}

\begin{proof}
By similar arguments as those in the proof of Lemma \ref{lem-signularity-decay}   or Remark \ref{rem-alt}, for $\varepsilon > 0$ small enough,  there exists $C>0$ such that for any  $z \in D_\varepsilon$, we have
$$
B_\omega^{\mathfrak{q}}(z,z)  \le C \prod_{i  =1}^k | z-q_i|^2,
$$
whence $| g(z)-1| B_\omega^{\mathfrak{q}}(z,z)$ is bounded on $D_\varepsilon$, and the first integral in \eqref{two-integrals} is bounded. 

For the second integral, the identities
$$
\left| \frac{z-p_j}{1 - \bar{p}_j z } \right|^2 = 1 - \frac{(1 - | z|^2) ( 1 - | p_j|^2) }{| 1 - \bar{p}_j z|^2},
$$
together with the same identities for $q_j: j = 1, \dots, k$, imply that there exists $C' > 0$ such that
$$
|g(z) -1|\le C'(1- |z|) \text{ for $z \in D_\varepsilon^c$.}
$$
Note also that since $\Ran (B_\omega^{\mathfrak{q}})\subset \Ran (B_\omega)$, we have 
$B_\omega^{\mathfrak{q}}(z,z) \le B_\omega(z,z)$, hence by our assumption \eqref{w-condition}, we have
$$
\int_{D_\varepsilon^c} | g(z)-1|^2 B_\omega^{\mathfrak{q}}(z,z) \omega(z)d\lambda(z) \le C' \int_{D_\varepsilon^c}(1-|z|)^2 B_\omega(z,z) \omega(z) d \lambda(z) < \infty.
$$
\end{proof}

\begin{lem}
The subspace $\sqrt{g} \cdot \Ran (B_\omega^{\mathfrak{q}}) $  is closed in  $L^2(\D, \omega d\lambda)$. Moreover,  for sufficiently large $R>0$, we have $\tr(\chi_{\{ g > R\} } [B_\omega^{\mathfrak{q}}]^g \chi_{\{ g>R\} }) <\infty$. 
\end{lem}

\begin{proof}
Note that  by Proposition \ref{prop-B-rel} and by defining a function $\alpha$ with $|\alpha(z)| =1$, given by
$$
\alpha(z)= \frac{|b_{\mathfrak{p}}(z) b_{\mathfrak{q}}(z)^{-1}|}{b_{\mathfrak{p}}(z) b_{\mathfrak{q}}(z)^{-1}}, \, z \in \D,
$$
we have 
\begin{align}\label{mod-1}
\sqrt{g} \cdot \Ran (B_\omega^{\mathfrak{q}}) = \alpha(z)  b_{\mathfrak{p}}(z) b_{\mathfrak{q}}(z)^{-1}  \mathscr{B}_\omega(\mathfrak{q})  =  \alpha(z) \mathscr{B}_\omega(\mathfrak{p}).
\end{align}
Since $ \mathscr{B}_\omega(\mathfrak{p})$ is a closed subspace in $L^2(\D, \omega d\lambda)$, so is  $\sqrt{g} \cdot \Ran (B_\omega^{\mathfrak{q}})$. By  \eqref{mod-1}, we have also  
$$
[B_\omega^{\mathfrak{q}}]^g = \alpha  \cdot B_\omega^{\mathfrak{p}}  \cdot  \overline{\alpha}.
$$
 It follows that, for sufficiently large $R>0$, since the set $\{z \in \D: g (z) > R\}$ is contained in a centered disk $\{z\in\D: |z| \le r\}$, with radius $r <1$, we have 
 \begin{align*}
\tr(\chi_{\{ g  > R\} }  [B_\omega^{\mathfrak{q}}]^g \chi_{\{ g>R\} })    = \tr(\chi_{\{ g  > R\} }  \alpha  \cdot B_\omega^{\mathfrak{p}}  \cdot  \overline{\alpha} \chi_{\{ g>R\} })      \le  \int_{|z| \le r}   B_\omega^{\mathfrak{p}}(z,z) \omega(z) d\lambda(z)< \infty. 
 \end{align*}
\end{proof}

\section{Proof of Theorem \ref{thm-C}}\label{proof-thm-C}

We start with an outline of our argument.

\begin{itemize} 
\item[(i)] Our first step is to define the regularized multiplicative functionals for functions $g$ such that the operator $\sqrt{|g-1|} \Pi \sqrt{|g -1|}$ belongs to the von Neumann-Schatten class $\mathscr{S}_3$.  In Definition \ref{def-A3}, we therefore introduce a class $\mathscr{A}_3(\Pi)$ of functions on $E$. We will see later in Proposition \ref{prop-log-A3}  and Proposition \ref{prop-general-reg} that if $g\in \mathscr{A}_3(\Pi)$, then the regularized multiplicative functional $\widetilde{\Psi}_g$  (cf. Definition \ref{def-mf-tilde}) is well-defined and integrable, the normalized multiplicative functional $\overline{\Psi}_g^\Pi$  is consequently well-defined.

\item[(ii)]  We prove in Proposition \ref{technical-prop} that if $g \in \mathscr{A}_3(\Pi)$ satisfies $\sup_E | g(x) - 1| < 1$, then the normalized generalized multiplicative functional $\overline{\Psi}_g^\Pi$ (see Definition \ref{def-mf-tilde}) is  well-defined and  the orthogonal projection onto the closed subspace $\sqrt{g}L$ induces a determinantal point process which coincides with $\overline{\Psi}_g^\Pi \PP_\Pi$.  The key step is the continuity, proved later in Proposition \ref{2-maps},  of the mapping that sends a function $g \in \mathscr{A}_3(\Pi)$  to $\overline{\Psi}_g^\Pi$. 

\item[(iii)] We derive Theorem \ref{thm-C} from Proposition \ref{technical-prop} by introducing a decomposition $g = g_1 g_2 g_3$ with $g_1\in \mathscr{A}_3(\Pi)$ that satisfies $\sup_E | g_1(x) - 1| < 1$ and  $g_2$ a compactly supported bounded function, $g_3$ a compactly supported function satisfying $g_3 >1$. The normalized regularized multiplicative functional $\overline{\Psi}_g^{\Pi}$ and the usual multiplicative functionals $\Psi_{g_2}$, $\Psi_{g_3}$ are all well-defined. We then write $\overline{\Psi}_{g}^\Pi = C \Psi_{g_3}\Psi_{g_2} \overline{\Psi}_{g_1}^{\Pi}$ and conclude the proof of Theorem \ref{thm-C}. 
\end{itemize}

\subsection{The class $\mathscr{A}_3(\Pi)$}
Recall that we denote by $\Pi$ an orthogonal projection on $L^2(E, \mu)$ which is locally in trace class. In \cite{BQS},  a class of Borel functions on $E$, denoted there by  $\mathscr{A}_2(\Pi)$,  plays a central role in the proof of the main result. By definition, $\mathscr{A}_2(\Pi)$ is the set of positive Borel functions $g$ on $E$ satisfying 
\begin{itemize}
\item[(1)] $0 < \inf\limits_E g \le \sup\limits_E g < \infty$; 
\item[(2)] $\int_E | g(x) - 1|^2 \Pi(x,x) d\mu(x) <\infty.$
\end{itemize}
If $g \in \mathscr{A}_2(\Pi)$, then the subspace $\sqrt{g}L$, where $L$ is the range of the orthogonal projection $\Pi$,  is automatically closed; we set $\Pi^g$ to be the corresponding operator of orthogonal projection.  The main property of $\mathscr{A}_2(\Pi)$ that will be used later is stated in the following
\begin{prop}[Cor. 4.2 of \cite{BQS}]\label{previous-prop}
If $g \in \mathscr{A}_2(\Pi)$ satisfies $\sup_E | g(x) - 1| < 1$, then the operator $\Pi^g$ is locally of trace class and the generalized multiplicative functional $\overline{\Psi}_{g}^{\Pi}$ are well-defined. Moreover, we have $\PP_{\Pi^g} = \overline{\Psi}_{g}^{\Pi}\cdot \PP_{\Pi}.$
\end{prop}

 Let $g: E \rightarrow \R$ be a Borel function, set
\begin{align}\label{def-L-f}
L(g): = \int_E | g(x) - 1|^3 \Pi(x,x) d\mu(x) \in [0, \infty]
\end{align}
and 
\begin{align}\label{def-V-f}
V(g): =   \iint_{E^2}    | g(x) - g(y) |^2   |\Pi(x,y)|^2 d \mu(x) d\mu(y) \in [0, \infty].
\end{align}
And then, we introduce a new class of Borel functions on $E$ as follows. 

\begin{defn}\label{def-A3}
 Let $\mathscr{A}_3(\Pi)$ be the set of positive Borel functions $g$ on $E$ satisfying 
\begin{itemize}
\item[(1)] $0 < \inf\limits_E g \le \sup\limits_E g < \infty$; 
\item[(2)] $L(g)  <\infty$ and $V(g) < \infty$;
\item[(3)] there exists an exhausting sequence $(E_n)_{n \ge 1}$ of bounded subsets of $E$, possibly depending on $g$,  such that 
\begin{align}\label{4-condition}
\lim_{n\to \infty}\tr (\chi_{E_n} \Pi | g-1|^2 \chi_{E_n^c} \Pi  \chi_{E_n}   ) = 0.
\end{align}
\end{itemize}
Moreover,  we introduce a topology $\mathscr{T}$ on $\mathscr{A}_3(\Pi)$ generated by the open sets
$$
U (\varepsilon, g) = \left\{ h \in \mathscr{A}_3(\Pi):  L(h/g) < \varepsilon, V(h/g) < \varepsilon \right\},
$$
In other words, a sequence $g_n$ converges to $g$ in $\mathscr{A}_3(\Pi)$ with respect to the topology $\mathscr{T}$ if and only if 
\begin{align}\label{tau-topo}
L(g_n/g) \to 0 \text{\, and \, } V(g_n/g) \to 0.
\end{align}

\end{defn}

Note that \eqref{4-condition} can equivalently be rewritten as
\begin{align}\label{4-condition-bis}
\lim_{n\to \infty} \iint_{E^2}  \chi_{E_n^c}(x) \chi_{E_n}(y) | g(x)-1|^2 | \Pi(x,y)|^2 d\mu(x) d\mu(y) =0.
\end{align}
 \begin{rem}
The sequence $(E_n)_{n\ge 1}$ in the definition of $\mathscr{A}_3(\Pi)$ is  an analogue of the sequence of the subsets $(\{z \in \C: | z| \le n\})_{n\ge 1}$ in the proof of Lemma \ref{lem-flatcut}. 
 \end{rem}

\begin{rem}
Note that the condition \eqref{4-condition} holds automatically for any $g \in \mathscr{A}_2(\Pi)$, hence we have 
$
\mathscr{A}_2(\Pi) \subset \mathscr{A}_3(\Pi) .
$
\end{rem}

\begin{rem}
Denote $[g, \Pi] := g \Pi - \Pi g$, then we have 
\begin{align}\label{commutator}
 V(g) = \| [g, \Pi]\|_{HS}^2,
 \end{align}
 where $\| \cdot \|_{HS}$ stands for the Hilbert-Schmidt norm. 
 \end{rem}

The main technical result in this section is the following 
\begin{prop}\label{technical-prop}
If $g \in \mathscr{A}_3(\Pi)$ satisfies $\sup_E | g(x) - 1| < 1$, then the operator $\Pi^g$ is locally of trace class and the generalized multiplicative functional $\overline{\Psi}_{g}^{\Pi}$ is well-defined. Moreover, we have $\PP_{\Pi^g} = \overline{\Psi}_{g}^{\Pi}\cdot \PP_{\Pi}$.
\end{prop}

\subsection{Derivation of Theorem \ref{thm-C} from Proposition \ref{technical-prop}}

 We now derive Theorem \ref{thm-C} from Proposition \ref{technical-prop}. The proof is similar to the proof of Proposition \ref{prop-BQS} given  in  \cite{BQS}. However, to prove the statement for $\mathscr{A}_3(\Pi)$ instead of $\mathscr{A}_2(\Pi)$ requires extra effort. 
 

By our assumption, we may choose $0<\varepsilon_1 < \varepsilon_2<1 $ and a {\it bounded} subset $E_{\mathrm{mid}} \subset E$, such that 
$$\{x\in E: | g(x)-1| \ge \varepsilon_2\} \subset E_{\mathrm{mid}} \subset \{x\in E: | g(x)-1| \ge \varepsilon_1\},
$$
 and the operator norm of $\chi_{\{x\in E: | g(x)-1| \le \varepsilon_2\}} \Pi$ is strictly less than $1$:
$$
\n \chi_{\{x\in E: | g(x)-1| \le \varepsilon_2\}} \Pi\n<1.
$$
Decompose $E_{\mathrm{mid}} = E_{\mathrm{mid}}^{+} \sqcup E_{\mathrm{mid}}^{-}$ by setting
$$
E_{\mathrm{mid}}^{+} = \{x\in E:  g(x) > 1 \} \cap E_{\mathrm{mid}} \et E_{\mathrm{mid}}^{-} = \{x\in E:  g(x) < 1 \} \cap E_{\mathrm{mid}}.
$$
Note that 
$$
E_{\mathrm{mid}}^{+} \subset \{x\in E: g(x) > 1+ \varepsilon_1\} \et E_{\mathrm{mid}}^{-} \subset \{x\in E: g(x) < 1-\varepsilon_1\}.
$$
Then we can decompose $g$ as $g = g_1 g_2 g_3$ with
\begin{align*}
g_1 &= (g-1)\chi_{E_{\mathrm{mid}}^c} +1,
\\
g_2& = (g-1)\chi_{E_{\mathrm{mid}}^{-}} + 1,
\\ 
g_3 & = (g-1)\chi_{E_{\mathrm{mid}}^{+}} + 1.
\end{align*}

\begin{claim}
 $g_1 \in \mathscr{A}_3(\Pi)$. 
\end{claim}
Indeed, the first two and the last conditions in the definition of $\mathscr{A}_3(\Pi)$ are immediate for $g_1$. We now check the third condition. We have
$$
| g_1(x) - g_1(y)| = 
\left\{
\begin{array}{cc} 
| g(x)- g(y)| & (x,y)\in E_{\mathrm{mid}}^c \times E_{\mathrm{mid}}^c
\\ | g(x)-1| & (x,y)\in E_{\mathrm{mid}}^c \times E_{\mathrm{mid}}
\\ | g(y)-1| & (x,y)\in E_{\mathrm{mid}} \times E_{\mathrm{mid}}^c
\\ 0 & (x,y)\times E_{\mathrm{mid}} \times E_{\mathrm{mid}}
\end{array},
\right.
$$
whence
\begin{align*}
V(g_1) =& \iint_{E^2}    | g_1(x) - g_1(y) |^2   |\Pi(x,y)|^2 d \mu(x) d\mu(y)
\\  = & \iint_{E_{\mathrm{mid}}^c\times E_{\mathrm{mid}}^c}    | g(x) - g(y) |^2   |\Pi(x,y)|^2 d \mu(x) d\mu(y)
\\ & +2   \int_{E_{\mathrm{mid}}}   d\mu(y) \int_{E_{\mathrm{mid}}^c} | g(x)-1|^2  | \Pi(x,y)|^2 d\mu(x).
\end{align*}
By \eqref{variance}, \eqref{flat-cut} and Remark \ref{HS-integral}, we have $V(g_1)< \infty.$

By Proposition \ref{technical-prop}, we have 
\begin{align}\label{g1-to-pi}
\PP_{\Pi^{g_1}} = \overline{\Psi}_{g_1}^{\Pi} \cdot \PP_{\Pi}.
\end{align}

The rest of the proof of Theorem \ref{thm-C}  follows the scheme of the proof of  Proposition \ref{prop-BQS} in \cite{BQS}. First, we have 
\begin{align}\label{3-g}
\Pi^{g_1g_2} = (\Pi^{g_1})^{g_2} \et \Pi^g = \Pi^{g_1g_2g_3} = (\Pi^{g_1g_2})^{g_3}.
\end{align}
Since $g_2$ is bounded and  $g_2-1$  is compactly supported, the usual multiplicative functional
$$
\Psi_{g_2}  (\XX) =  \prod_{x\in \XX} g_2(x),
$$
is well-defined and an application of the main result in \cite{Buf-multi} yields that 
\begin{align}\label{g12-to-g1}
\PP_{\Pi^{g_1g_2}} = C_1  \Psi_{g_2} \PP_{\Pi^{g_1}}.
\end{align}
The function $g_3-1$, although not necessarily bounded, is compactly supported and positive.  The usual multiplicative functional 
$\Psi_{g_3}$ is also well-defined for  $\PP_{\Pi^{g_1g_2}}$-almost every configuration. Indeed, since $g_1g_2$ is bounded and by \cite[Prop. 4.4]{BQS}, there exists $C>0$ such that 
$$
\Pi^{g_1g_2}(x,x) \le C \Pi(x,x).
$$
Consequently, we have
\begin{align}\label{support-g3}
\int_E |g_3(x)-1|\Pi^{g_1g_2}(x,x) d\mu(x)\le  C \int_{E_{\mathrm{mid}}^{+}} |g_3(x)-1|\Pi (x,x) d\mu(x) <\infty.
\end{align}
In the relation \eqref{support-g3}, we used the fact that $g_3 -1$ is supported on $E_{\mathrm{mid}}^{+}$ and our assumption \eqref{singularity}. 
It follows that     $\sqrt{g_3-1}\Pi^{g_1g_2} $ is Hilbert-Schmidt. By definitinon
$$
\sqrt{g_3} \cdot \Ran( \Pi^{g_1g_2}) = \sqrt{g_2} \sqrt{g_1g_2} L = \sqrt{g} L,
$$
hence by assumption of Theorem \ref{thm-C},  $\sqrt{g_3} \cdot \Ran( \Pi^{g_1g_2})$  is closed.  Moreover, by \eqref{3-g}, we have $(\Pi^{g_1 g_2})^{g_3} = \Pi^{g}$, which is locally of trace class by assumption.  For $R$ large enough, $g_3 > R$ implies $g> R$, hence by assumption \eqref{extra}, we have
$$
\tr(\chi_{\{ g_3 > R\} }  (\Pi^{g_1 g_2})^{g_3}\chi_{\{ g_3> R\} }) = \tr(\chi_{\{ g_3 > R\}}  \Pi^{g}\chi_{\{ g_3> R \} }) \le  \tr(\chi_{\{ g > R\} }  \Pi^{g}\chi_{\{ g> R\} }) < \infty. 
  $$
By  \cite[Prop. 4.10]{BQS}, we have 
\begin{align}\label{g-to-g12}
\PP_{\Pi^g} = C' \Psi_{g_3}\PP_{\Pi^{g_1g_2}}.
\end{align}
Combining \eqref{g1-to-pi}, \eqref{g12-to-g1} and \eqref{g-to-g12}, we get 
\begin{align}\label{mul-reg}
\PP_{\Pi^g} = C' \Psi_{g_3}\PP_{\Pi^{g_1g_2}} = C'C \Psi_{g_3} \Psi_{g_2} \cdot  \PP_{\Pi^{g_1}} = C'C\Psi_{g_3}\Psi_{g_2} \overline{\Psi}_{g_1}^{\Pi} \cdot\PP_{\Pi},
\end{align}
whence $\PP_{\Pi^g}  = \overline{\Psi}^{\Pi}_g \PP_\Pi$ and Theorem \ref{thm-C} is completely proved. 

\begin{rem}\label{rem-mul}
The following elementary observation is used in the equality \eqref{mul-reg}: if $g, h$ are two non-negative functions such that for $h-1$ is compactly supported and $\overline{\Psi}_g^{\Pi}$ is defined. Then the usual multiplicative functional $\Psi_h$ is well-defined: $\Psi_h(X)  = \prod_{x\in X} h(x)$. 
Moreover,  the regularized multiplicative functional  $\overline{\Psi}_{gh}^{\Pi} $ is well-defined and there exists a unique $C>0$, such that 
\begin{align*}
\overline{\Psi}_{gh} = C \Psi_h \cdot \overline{\Psi}_{g}. 
\end{align*}
Indeed, we have 
\begin{align*}
\log \widetilde{\Psi}_{gh} (X)  =& \lim_{n\to\infty} \Big( \sum_{x\in E_n} \log g(x) h(x)  - \E_{\PP_{\Pi}}  \sum_{x\in E_n} \log g(x) h(x) \Big) 
\\
 =& \lim_{n\to\infty} \Big( \sum_{x\in E_n} \log g(x)  - \E_{\PP_{\Pi}}  \sum_{x\in E_n} \log g(x) \Big) 
 \\
 & +  \lim_{n\to\infty} \Big( \sum_{x\in E_n} \log h(x)  - \E_{\PP_{\Pi}}  \sum_{x\in E_n} \log h(x) \Big) 
 \\
 = &  \log \widetilde{\Psi}_{g} (X)  + \sum_{x \in X} \log h(x) + \log C_1 = \log (C_1\widetilde{\Psi}_g(X) \Psi_h(X)).
\end{align*}
That is, $\widetilde{\Psi}_{gh}  = C_1\widetilde{\Psi}_g \Psi_h$.  It follows that 
\begin{align*}
\overline{\Psi}_{gh} = \frac{\widetilde{\Psi}_{gh} }{\E_{\PP_{\Pi}} [\widetilde{\Psi}_{gh}]}  =  \frac{\widetilde{\Psi}_g \Psi_h}{ \E_{\PP_{\Pi}} [\widetilde{\Psi}_g \Psi_h]} =   \frac{ \E_{\PP_{\Pi}} [\widetilde{\Psi}_g ] }{ \E_{\PP_{\Pi}} [\widetilde{\Psi}_g \Psi_h]} \cdot \frac{\widetilde{\Psi}_g}{\E_{\PP_{\Pi}} [\widetilde{\Psi}_g ]}   \Psi_h = C \Psi_h \overline{\Psi}_g,
\end{align*}
where $C = \frac{ \E_{\PP_{\Pi}} [\widetilde{\Psi}_g ] }{ \E_{\PP_{\Pi}} [\widetilde{\Psi}_g \Psi_h]}$ is uniquely determined.
\end{rem}

 \subsection{Convergence in $\mathscr{A}_3(\Pi)$}
We need the following  convergence properties of functions in $\mathscr{A}_3(\Pi)$.
 
 \begin{lem}\label{lem-cut}
Let $g \in \mathscr{A}_3(\Pi)$ and let $(E_n)_{n \ge 1}$ be the exhausting sequence of bounded subsets of $E$ such that condition \eqref{4-condition} holds. Denote $g_n = 1 + (g-1)\chi_{E_n}$, then $g_n \xrightarrow[n\to \infty]{\mathscr{T}} g.$
 \end{lem}
 \begin{proof}
Assume that $g\in \mathscr{A}_3(\Pi)$. First, by definition, we have
 $$
 | g_n/g -1| = | 1/g -1|\chi_{E_n^c}\le \frac{1}{\inf_E g} | g-1|.
 $$
 It follows that $L(g_n/g)\to 0$. 
 
 Next, setting 
 $$
 V_n(x,y) = |g_n(x) /g (x) - g_n(y)/g (y)|^2 | \Pi(x,y)|^2,
 $$
 we have 
 \begin{align}\label{3-term}
 V(g_n/g) = \iint\limits_{E_n \times E_n^c} V_n+ \iint\limits_{E_n^c \times E_n} V_n + \iint\limits_{E_n^c\times E_n^c} V_n.
 \end{align}
 The first and second terms in \eqref{3-term} are equal and 
 \begin{align*}
 \iint\limits_{E_n \times E_n^c} V_n &= \iint\limits_{E_n \times E_n^c} | 1 - 1/g(y)|^2 | \Pi(x,y)|^2  d\mu(x) d\mu(y) \\ & \le \frac{1}{\inf_E g^2}   \iint\limits_{E_n \times E_n^c} | g(y) -1|^2 | \Pi(x,y)|^2  d\mu(x) d\mu(y) \\  & = \frac{1}{\inf_E g^2}  \| \chi_{E_n}  \Pi  | g-1| \chi_{E_n^c}  \|_2^2\to 0.
 \end{align*}
The third term in \eqref{3-term} converges to 0 since 
 $$
  \iint\limits_{E_n^c\times E_n^c} V_n \le \frac{1}{\inf_E g^2}\iint\limits_{E_n^c\times E_n^c} | g(x)- g(y) |^2 | \Pi(x,y)|^2 d\mu(x) d\mu(y),
 $$
 and the latter integral tends to $0$ as $n\to \infty$.  Thus $V(g_n/g) \to 0$,  and  Lemma \ref{lem-cut} is completely proved.
 \end{proof}

\begin{lem}\label{LV-conv}
Let $g_n \in \mathscr{A}_3(\Pi), n \ge 1$, $g\in \mathscr{A}_3(\Pi)$, and assume that the sequence $(g_n)$ is uniformly bounded. 
If  $g_n \xrightarrow[n \to \infty]{\mathscr{T}} g$, then $L(g_n) \to L(g)$ and $V(g_n) \to V(g).$
\end{lem}
\begin{proof}
By definition, we have $L(g_n/g)\to 0$ and $V(g_n/g)\to 0$. 
 
The relation  $L(g_n/g)\to 0$ together with the inequality 
 $$
 \int | g_n(x)-g(x)|^3 \Pi(x,x) d\mu(x) \le  \sup_E g \cdot \int | g_n(x)/g(x) -1|^3 \Pi(x,x)d\mu(x) 
 $$
 implies that 
 $$
\lim_{n\to \infty} \n (g_n-1) - (g-1)\n_{L^3(E; \Pi(x,x)d\mu(x))}   = 0,
 $$
 whence 
 $$
 \lim_{n\to \infty}\n g_n-1\n_{L^3(E; \Pi(x,x)d\mu(x)} =\n g-1 \n_{L^3(E; \Pi(x,x)d\mu(x))}.
 $$
 This is equivalent to $L(g_n) \to L(g)$ as $n\to \infty$.
 
We turn to the proof of the convergence $V(g_n) \to V(g)$. It suffices to prove any convergent subsequence (in $[0, \infty]$) of the sequence $(V(g_n))_{n \ge 1}$ converges to $V(g)$. We have already shown  that 
$$
\intt_E | g_n(x) - g(x)|^3 \Pi(x,x) d\mu(x) \to 0.
$$ 
Passing perhaps to a subsequence, we may assume that $g_n \to g$ almost everywhere with respect to $\Pi(x,x) d \mu(x)$. 
Set 
$$
F_n(x,y) = g_n(x) - g_n(y) \et F(x,y) = g(x) - g(y).
$$
The desired relation $V(g_n) \to V(g)$ is equivalent to the relation
$$
\lim_{n\to \infty}\n F_n\n_{L^2(E\times E; \, | \Pi(x,y)|^2 d\mu(x)d\mu(y))} = \n F\n_{L^2(E\times E; \, | \Pi(x,y)|^2 d\mu(x)d\mu(y))}
$$
To simplify notation, write $dM_2(x,y) = | \Pi(x,y)|^2 d\mu(x)d\mu(y)$. It suffices to prove that 
\begin{align}\label{V-con}
\lim_{n\to \infty} \n F_n - F \n_{L^2(E\times E; \, dM_2)} = 0.
\end{align}
A direct computation shows that
\begin{align*}
 \frac{ F_n(x,y) - F(x,y)}{g(x)}   = \frac{g_n(x)}{g(x)}  -  \frac{g_n(y)}{g(y)} + \frac{F(x,y) (g_n(y) - g(y))}{g(x)g(y)}.
\end{align*}
Hence we have 
\begin{align*}
 | F_n(x,y) - F(x,y) |  \le \sup_E g \cdot  \left| \frac{g_n(x)}{g(x)} - \frac{g_n(y)}{g(y)}  \right|  +  \frac{1}{\inf_E g} | F(x,y)| \cdot | g_n(y) - g(y)|,
\end{align*}
and 
\begin{align*}
 \n F_n - F \n_{L^2(E\times E;  \, dM_2)} & \le \sup_E g \cdot  \left\n \frac{g_n(x)}{g(x)} - \frac{g_n(y)}{g(y)}  \right\n_{L^2(E\times E; \, dM_2)}  
\\ & +  \frac{1}{\inf_E g} \left\n F(x,y) \cdot | g_n(y) - g(y)| \right\n_{L^2(E\times E; \, dM_2)} 
\end{align*}
The limit relation $V(g_n/g) \to 0$ implies that 
$$
\lim_{n\to \infty}  \left\n \frac{g_n(x)}{g(x)} - \frac{g_n(y)}{g(y)}  \right\n_{L^2(E\times E; \, dM_2)}   = 0. 
$$
By definition,  $F \in L^2(E\times E; \, dM_2)$. Since the sequence $(g_n)$ is uniformly bounded and $g_n \rightarrow g$ almost everywhere with respect to $\Pi(x,x)d\mu(x)$,  the  dominated convergence theorem yields
$$
\lim_{n\to \infty}  \left\n F(x,y) \cdot | g_n(y) - g(y)| \right\n_{L^2(E\times E; \, dM_2)}  = 0.
$$
This completes the proof of \eqref{V-con}. Lemma \ref{LV-conv} is proved completely.
\end{proof}

\subsection{Existence of generalized multiplicative functionals}

Recall that,  in Definition \ref{d-v0} and Definition \ref{def-mf-tilde}, we introduced the subset $\VV_0(\Pi) \subset \VV(\Pi)$ and the functional $\widetilde{\Psi}_g$ for functions $g$ such that $\log g \in \VV_0(\Pi)$. Recall also that we introduced in \eqref{var-pi-f} the notation $\Var(\Pi, f)$ for any Borel function $f : E \rightarrow \C$. 

\begin{prop}\label{prop-log-A3}
If $g \in \mathscr{A}_3(\Pi)$, then $\Var(\Pi, \log g ) < \infty$ and $\log g \in \VV_0(\Pi)$. 
In particular, for any function $g \in \mathscr{A}_3(\Pi)$, the functional $\widetilde{\Psi}_g$ is well-defined. 
\end{prop}

\begin{proof}
By the third condition in the definition of $\mathscr{A}_3(\Pi)$, if $g \in \mathscr{A}_3(\Pi)$, then $$\Var (\Pi, g -1) < \infty.$$ 
Define a function 
$$ 
F(t) := \left\{ \begin{array}{cc} \frac{\log (1 + t) - t }{t^2} & \text{if $t \ne 0$} \\ 
 -\frac{1}{2} & \text{if $t =0$} \end{array} \right..
$$
Then $F$ is continuous on $(-1, \infty)$. It follows that for any $0<\varepsilon \le1$ and  $M \ge 1$, there exists $C_{\varepsilon,  M}>0$, such that if $t \in [-1 + \varepsilon, - 1 + M]$, then
\begin{align}\label{log-expansion}
\left| \log (1 +t) - t  \right| \le C_{\varepsilon, M} t^2.
\end{align}
By the first condition in the definition of $\mathscr{A}_3(\Pi)$, we can apply the above inequality to $g-1$. A simple computation yields
\begin{align}\label{log-difference}
\begin{split}\left|  \log g(x) - \log g(y)  \right|^2   \le  &20M^2 | g(x) - g(y)|^2 
\\ & + 8 M C_{\varepsilon, M}^2 (| g(x)-1|^3 + | g(y)-1|^3),
\end{split}
\end{align}
where $\varepsilon = \min (1, \inf_E g)$ and $M = \max(1, \sup_E g)$. Inequality \eqref{log-difference},  combined with the reproducing property of the kernel $\Pi$: 
$$
\Pi(x,x) = \int_E | \Pi(x,y)|^2 d\mu(y)
$$
and the second and third conditions on $g$ in the definition of $\mathscr{A}_3(\Pi)$, yields the desired inequality $\Var(\Pi, \log g) < \infty.$

We turn to the proof of the relation $\log g \in \VV_0(\Pi)$. By definition, there exists a sequence $(E_n)$ of exhausting bounded subsets of $E$, such that the relation \eqref{4-condition-bis} holds. It suffices to show that 
\begin{align}\label{to-do-1}
\lim_{n\to \infty}\n \chi_{E_n}\log g -  \log g \n_{\VV(\Pi)} = \lim_{n\to \infty} \n \chi_{E_n^c} \log g \n_{\VV(\Pi)} = 0.
\end{align}
We have
\begin{align*}
\n \chi_{E_n^c} \log g \n_{\VV(\Pi)}^2 
= & \frac{1}{2}  \iint_{E_n^c \times E_n^c}  | \log g(x) - \log g(y)|^2 | \Pi(x,y)|^2 d\mu(x) d\mu(y)
\\ &  + \frac{1}{2} \iint_{E^2} \chi_{E_n^c}(x) \chi_{E_n}(y) | \log g(x)|^2 | \Pi(x,y)|^2d\mu(x) d\mu(y)
\\ &+ \frac{1}{2} \iint_{E^2} \chi_{E_n^c}(y) \chi_{E_n}(x) | \log g(y)|^2 | \Pi(x,y)|^2d\mu(x) d\mu(y).
\end{align*}
Since $\Var(\Pi, \log g) < \infty$, the  first integral in the above identity tends to $0$ when $n$ tends to infinity follows.  The second and the third integrals are equal, and since $\varepsilon \le g \le M$, we may use $|\log g(x) | \le C_{\varepsilon, M} | g(x)-1|$ and we get
\begin{align}\label{2-int}
\begin{split}
&\iint_{E^2} \chi_{E_n^c}(x) \chi_{E_n}(y) | \log g(x)|^2 | \Pi(x,y)|^2d\mu(x) d\mu(y)
\\ 
\le &  
C_{\varepsilon, M}^2 \iint_{E^2} \chi_{E_n^c}(x) \chi_{E_n}(y) | g(x) - 1 |^2 | \Pi(x,y)|^2d\mu(x) d\mu(y).
\end{split}
\end{align} 
The assumption \eqref{4-condition-bis} implies that the last integral in \eqref{2-int} tends to $0$ as $n$ tends to infinity. This completes the proof of  the desired relation \eqref{to-do-1}. 
\end{proof}

\begin{defn}
Let $\mathscr{A}_3^{\varepsilon, M}(\Pi) \subset \mathscr{A}_3(\Pi) $ be the subset of functions such that
\begin{align}\label{future-reference}
\varepsilon \le \inf\limits_E g \le  \sup\limits_E g  \le M.
\end{align}
\end{defn}

\begin{prop}\label{prop-general-reg}
For any $\varepsilon, M:$ $0< \varepsilon \le 1$, $M \ge1$, there exists a constant $C_{\varepsilon, M}> 0$ such that if $g \in \mathscr{A}_3^{\varepsilon, M}(\Pi)$,
then 
\begin{align}\label{L1}
\log \E \widetilde{\Psi}_g \le C_{\varepsilon, M}  ( L(g)+V(g) ).
\end{align}
In particular, the normalized generalized multiplicative functional $\overline{\Psi}_g^\Pi$ is well-defined. 
\end{prop}

Denote 
$
g^{+} =  1 + \chi_{\{ g \ge 1\} } (g-1)  \et g^{-} = 1 + \chi_{\{ g \le 1\}} (g-1).
$
Then  $g = g^{+} g^{-}$ with $g^{+} \ge 1, g^{-} \le 1$. Our aim here is to reduce Proposition \ref{prop-general-reg} for $g$ to the same statement for $g^{+}, g^{-}$.
\begin{lem}\label{pm}
Both $g^{+}$ and $g^{-}$ are in the class $\mathscr{A}_3^{\varepsilon, M}(\Pi)$, moreover,   we have 
\begin{align}\label{LV}
L(g^{\pm}) \le L(g) \et V(g^{\pm}) \le V(g).
\end{align}
\end{lem}

\begin{proof}
Inequalities \eqref{LV} follow from the elementary inequalities 
\begin{align}\label{flat}
| g^{\pm} - 1| \le | g - 1| \et | g^{\pm} (x) - g^{\pm}(y) | \le | g(x) - g(y)|.
\end{align}
Now let  $(E_n)_{n\ge 1}$ be the exhausting sequence of bounded subsets such that \eqref{4-condition} holds.  The first inequality in \eqref{flat} yields the following inequalities for self-adjoint operators: 
\begin{align*}
\chi_{E_n} \Pi | g^{\pm}-1|^2 \chi_{E_n^c} \Pi  \chi_{E_n}  \le   \chi_{E_n} \Pi | g-1|^2 \chi_{E_n^c} \Pi  \chi_{E_n}.
\end{align*}
Hence \eqref{4-condition} holds for $g^{\pm}$ with respect to the sequence $(E_n)_{n\ge 1}$. Consequently, $g^{\pm}$ are both in $\mathscr{A}_3^{\varepsilon, M}(\Pi)$.
\end{proof}

Denote by $\mathscr{A}_3^{\varepsilon, M}(\Pi)^{+}$  the subclass of  functions in $\mathscr{A}_3^{\varepsilon, M}(\Pi)$ such that 
$$
g \in \mathscr{A}_3(\Pi) \et  g \ge 1 .
$$ Similarly, denote by $\mathscr{A}_3^{\varepsilon, M}(\Pi)^{-}$  the subclass of  functions in $\mathscr{A}_3^{\varepsilon, M}(\Pi)$ such that 
$$ 
g \in \mathscr{A}_3^{\varepsilon, M}(\Pi) \et   g \le 1. 
$$
Let 
$$
\mathscr{A}_3^{\varepsilon, M}(\Pi)^{\pm} = \mathscr{A}_3^{\varepsilon, M}(\Pi)^{+} \cup \mathscr{A}_3^{\varepsilon, M}(\Pi)^{-}.
$$
 
We reduce the statement of Proposition \ref{prop-general-reg} for general $g \in \mathscr{A}_3^{\varepsilon, M}(\Pi)$ to the particular case $g \in \mathscr{A}_3^{\varepsilon, M}(\Pi)^{\pm}$. Indeed, assume that we have established \eqref{L1} in the case of $\mathscr{A}_3^{\varepsilon, M}(\Pi)^{\pm}$, then by multiplicativity, for general $g$ in $\mathscr{A}_3^{\varepsilon, M}(\Pi)$, we have 
\begin{align*}
 \E \widetilde{\Psi}_{g} & 
 = \E (\widetilde{\Psi}_{g^{+}} \widetilde{\Psi}_{g^{-}}) \le  (\E \widetilde{\Psi}_{g^{+}}^2 \cdot \E \widetilde{\Psi}_{g^{-}}^2)^{1/2}
  = (\E \widetilde{\Psi}_{(g^{+})^2} \cdot \E \widetilde{\Psi}_{(g^{-})^2})^{1/2} 
  \\
  & \le \frac{1}{2} (\E \widetilde{\Psi}_{(g^{+})^2} + \E \widetilde{\Psi}_{(g^{-})^2}).
\end{align*}
Now we may apply \eqref{L1} for functions $(g^{+})^2 \in \mathscr{A}_3^{\varepsilon, M}(\Pi)^{+}$ and $(g^{-})^2 \in \mathscr{A}_3^{\varepsilon, M}(\Pi)^{-}$ respectively and use the relations \eqref{semi-mul} together with Lemma \ref{pm} , to obtain that 
\begin{align*}
 \E \widetilde{\Psi}_g & \le C'\Big[L((g^{+})^2) + V((g^{+})^2) + L((g^{-})^2) + V((g^{-})^2  \Big]
 \\ & \le C'' \Big[L(g^{+}) + V(g^{+}) + L(g^{-}) + V(g^{-}))\Big]
 \\ & \le C''' (L(g) + V(g)).
\end{align*}

We  now proceed to the proof of \eqref{L1} for functions $g$ in $\mathscr{A}_3^{\varepsilon, M}(\Pi)^{\pm}$ and, consequently, Proposition \ref{prop-general-reg}.  By definition, if  $g \in \mathscr{A}_3^{\varepsilon, M}(\Pi)^{\pm}$, then the sequences $(g_n)_{n\ge 1}$ defined in the proof of Lemma \ref{lem-cut} all stay in the set $\mathscr{A}_3^{\varepsilon, M}(\Pi)^{\pm}$. Note that by the computation in \eqref{to-do-1}, we have 
\begin{align*}
& \| \overline{S}_{\log g_n} - \overline{S}_{\log g} \|_2^2 = \Var (\Pi, \log (g_n/g) )  =   \Var (\Pi,  \chi_{E_n^c} \log g ) = 2 \n \chi_{E_n^c} \log g \n_{\VV(\Pi)}^2  \xrightarrow{n\to\infty} 0.
\end{align*}
Consequently, passing perhaps to a subsequence, we may assume that  
$$
\overline{S}_{\log g_n} \xrightarrow[n\to \infty]{a.e.} \overline{S}_{\log g}  
$$ 
and hence  
$$ 
\widetilde{\Psi}_{g_n}   = \exp(\overline{S}_{\log g_n})  \xrightarrow[n\to \infty]{a.e.} \widetilde{\Psi}_{g} = \exp ( \overline{S}_{\log g}) .
$$ 
By Fatou's Lemma and Lemma \ref{LV-conv} , it suffices to establish \eqref{L1} for a function $g \in \mathscr{A}_3^{\varepsilon, M}(\Pi)^{\pm}$ such that the subset $\{x \in E: g(x)  \ne 1\}$ is bounded. We will assume the boundedness of $\{x \in E: g(x) \ne 1\}$ until the end of the proof of Proposition \ref{prop-general-reg}.

For any $0< \varepsilon \le 1$ and any $M\ge1$, there exists $C_{\varepsilon, M} > 0$ such that if $t \in [ -1 + \varepsilon, -1 + M]$, then
\begin{align} \label{log}
\left| \log (1+t) - t  + \frac{1}{2}t^2   \right| \le C_{\varepsilon, M} \cdot |t|^{3}.
\end{align}

Recall that for any bounded linear operator $A$ acting on a Hilbert space, we set $|A| =\sqrt{A^*A}$.  The inequality \eqref{log} applied to the eigenvalues of trace class operator with spectrum contained in $ [ -1 + \varepsilon, -1 + M]$ yields the following
\begin{lem}
Let $\varepsilon, M, C_{\varepsilon,M}$ be as in the inequality \eqref{log}. For any {\it self-adjoint} trace class operator $A$ whose spectrum  $\sigma(A)$ satisfies $\sigma(A) \subset [ - 1 + \varepsilon, - 1 + M]$, we have
\begin{align}\label{log-det}
\log \det (1 + A) \le    \tr (A ) - \frac{1}{2} \tr (A^2)  + C_{\varepsilon, M} \tr(|A|^{3}).
\end{align}
\end{lem}

\begin{proof}
The lemma is an immediate consequence of the inequality \eqref{log} and the identity
\begin{align*}
\log \det ( 1+ A) = \sum_{i = 1}^\infty \log (1 + \lambda_i(A)),
\end{align*}
where $(\lambda_i(A))_{i =1}^\infty$ is the sequence of the eigenvalues of $A$.
\end{proof}

In order to simplify notation, for $g \in   \mathscr{A}_3^{\varepsilon, M}(\Pi)^{+}$, set
\begin{align}\label{h+}
\text{$h = g-1\ge 0$ and $T_g^{+} = \sqrt{h} \Pi \sqrt{h} \ge 0$;}
\end{align}
and for $g\in \mathscr{A}_3^{\varepsilon, M}(\Pi)^{-}$, set
\begin{align}\label{h-}
\text{$h =  g-1 \le 0$ and  $T_g^{-} = \Pi h \Pi \le 0$.}
\end{align}

Applying the relation \eqref{log-det}, for $g \in \mathscr{A}_3^{\varepsilon, M}(\Pi)^{\pm}$, we have
\begin{align}\label{TF1}
\begin{split}
 \log \E \Psi_g &  = \log \det ( 1 + (g-1) \Pi)    = \log \det ( 1 + T_g^{\pm}) \\  &  \le     \tr(T_g^{\pm}) - \frac{1}{2} \tr((T_g^{\pm})^2)   + C_{\varepsilon, M} \tr(|T_g^{+}|^{3}).
 \end{split}
\end{align}
Clearly, the traces $\tr(T_g^{+})$ and $\tr(T_g^{-})$  are given by the formula: 
\begin{align}
\tr(T_g^{\pm}) = \intt_E h(x) \Pi(x,x) d\mu(x). 
\end{align}

Recall that the  inner product on the space of Hilbert-Schmidt operators is defined by the formula
$$
\langle a, b \rangle_{HS} = \tr(a b^{*}).
$$

\begin{lem}\label{square}
For any $g \in \mathscr{A}_3^{\varepsilon, M}(\Pi)^{\pm}$, we have 
\begin{align}\label{tg}
 \tr((T_g^{\pm})^2)  =   \intt_E h(x)^2 \Pi(x,x) d\mu(x)  - \frac{1}{2} V(g).
\end{align}
\end{lem}

\begin{proof}
If  $g \in \mathscr{A}_3^{\varepsilon, M}(\Pi)^{+}$, then 
\begin{align}\label{product}
\tr ((T_g^{+})^2) & = \tr(\sqrt{h} \Pi h \Pi\sqrt{h}) =   \tr(\Pi h \Pi h)   = \langle \Pi h, h\Pi      \rangle_{HS}.
\end{align}
Note that 
\begin{align}\label{HS}
\| \Pi h \|_{HS}^2 = \| h  \Pi \|_{HS}^2 = \intt_E h(x)^2 \Pi(x,x) d\mu(x).
\end{align}
By \eqref{commutator}, we have 
\begin{align}\label{VG}
\begin{split}
V(g) &= \| [g, \Pi]\|_{HS}^2 = \| [h, \Pi]\|_{HS}^2  = \| h \Pi- \Pi h\|_{HS}^2  \\ & =  \| h \Pi \|_{HS}^2 +   \|  \Pi  h \|_{HS}^2 - 2 \langle h \Pi, \Pi h  \rangle.
\end{split}
\end{align}
Combining \eqref{product}, \eqref{HS} and \eqref{VG}, we complete the proof of the desired identity \eqref{tg} for $g  \in \mathscr{A}_3^{\varepsilon, M}(\Pi)^{+}$. 

The argument for $g  \in \mathscr{A}_3^{\varepsilon, M}(\Pi)^{-}$ is the same, since we have 
$$
\tr ((T_g^{-})^2) = \tr (\Pi f \Pi f \Pi) = \tr (\Pi f \Pi f). 
$$
\end{proof}

\begin{lem}
For any $g \in \mathscr{A}_3^{\varepsilon, M}(\Pi)^{\pm}$, we have
\begin{align}\label{Tg3}
\tr(| T_g^{\pm}|^3) \le  L(g) = \intt_E  | g(x)-1|^3 \Pi(x,x) d\mu(x).
\end{align}
\end{lem}

\begin{proof}
First, let $g \in \mathscr{A}_3^{\varepsilon, M}(\Pi)^{+}$. Recall the definition of $h$ and $T_g^{+}$ in \eqref{h+}. By the elementary operator inequality
$$
 \sqrt{h} \Pi h \Pi h \Pi \sqrt{h} \le \sqrt{h}\Pi h^2 \Pi \sqrt{h},
$$
we get 
\begin{align}\label{power3}
\tr(|T_g^{+}|^3) =  \tr (\sqrt{h} \Pi h \Pi h\Pi \sqrt{h})  \le \tr (\sqrt{h}\Pi h^2 \Pi \sqrt{h}) = \| \sqrt{h}\Pi h\|_{HS}^2.
\end{align}
Since 
\begin{align*}
\| \sqrt{h}\Pi h\|_{HS}^2 & = \tr (\sqrt{h}\Pi h^2 \Pi \sqrt{h}) = \tr (\Pi h^{3/2}  h^{1/2} \Pi h)    = \langle \Pi h^{3/2}, h \Pi h^{1/2}\rangle_{HS} \\ & \le \|\Pi h^{3/2}\|_{HS} \| h\Pi h^{1/2}\|_{HS} =  \|\Pi h^{3/2}\|_{HS} \| \sqrt{ h}\Pi h\|_{HS},
\end{align*}
we also have 
\begin{align}\label{power3-bis}
\| \sqrt{h}\Pi h \|_{HS}^2 \le \|\Pi h^{3/2}\|_{HS}^2 = \tr (\Pi h^3\Pi) = \tr(h^3 \Pi)  = L(g).
\end{align}
Combining inequalities \eqref{power3} and \eqref{power3-bis}, we obtain the desired inequality \eqref{Tg3} for $g \in  \mathscr{A}_3^{\varepsilon, M}(\Pi)^{+}$.

The inequality \eqref{Tg3} for $g \in  \mathscr{A}_3^{\varepsilon, M}(\Pi)^{-}$ is proved by noting that in this case, $T_g^{-} =   \Pi h \Pi =-  \Pi | h | \Pi $ and 
\begin{align*}
\tr (|T_g^{-}|^3)&= \tr( \Pi |h| \Pi | h| \Pi |h| \Pi) = \tr(  \sqrt{| h|} \Pi |h| \Pi |h| \Pi \sqrt{|h|}) \\ & \le \tr (\sqrt{| h|} \Pi |h|^2 \Pi \sqrt{|h|}).
\end{align*}
\end{proof}

\begin{proof}[Conclusion of the proof of Proposition \ref{prop-general-reg}]
It suffices to establish \eqref{L1} when $g \in \mathscr{A}_3^{\varepsilon, M}(\Pi)^{\pm}$. An application of \eqref{log} yields that 
\begin{align}\label{g-h}
\begin{split}
&\left| \int\limits_E    \left( \log g(x)     -  h(x) + \frac{h(x)^2}{2} \right)   \Pi(x,x) d\mu(x) \right|  \le  C_{\varepsilon, M} L(g).
\end{split}
\end{align}
It follows that
\begin{align*}
\log \E\widetilde{\Psi}_g  =& \log \E \Psi_g - \E S_{\log g}  \\  \le&  \tr (T_g^{\pm}) - \frac{1}{2} \tr ((T_g^{\pm})^2) + C_{\varepsilon, M} \tr ( |T_g^{\pm}|^3) - \E S_{\log g} \\ \le &  \intt_E h(x) \Pi(x,x) d\mu(x)  - \frac{1}{2} \intt_E h(x)^2 \Pi(x,x) d\mu(x) + \frac{1}{4} V(g) \\ & + C_{\varepsilon,M} L(g) -  \intt_E \log g(x) \Pi(x,x) d\mu(x)  \\ \le &   2 C_{\varepsilon, M}    L(g) + \frac{1}{4}V(g)  = C_{\varepsilon,M}' (L(g) + V(g)).
\end{align*}
\end{proof}

\subsection{Continuity and convergence of regularized multiplicative functionals}

\begin{prop}\label{general-reg}
For any $\varepsilon, M:$ $0< \varepsilon \le 1$, $M \ge1$, there exists a constant $C_{\varepsilon, M}> 0$ such that  if $g \in \mathscr{A}_3^{\varepsilon, M}(\Pi)$, then 
\begin{align}\label{L2}
\log \E |\widetilde{\Psi}_g|^2 \le C_{\varepsilon, M}  ( L(g) +V(g) ).
\end{align}
\end{prop}

\begin{proof}
 By definition $| \widetilde{\Psi}_g|^2 = \widetilde{\Psi}_{g^2}$. If $g \in \mathscr{A}_3^{\varepsilon, M}(\Pi)$, then 
  \begin{align}\label{semi-mul}
 L(g^2) \le 8M^3 L(g) \et V(g^2) \le 4M^2 V(g) .
 \end{align}
Consequently, Lemma follows immediately from  the estimate \eqref{L1} in Proposition \ref{prop-general-reg}.
\end{proof}

\begin{prop}\label{L1-prop}
Given $0< \varepsilon \le 1$ and $M \ge1$, there exists a constant $C_{\varepsilon, M}> 0$ such that if $g_1, g_2 \in \mathscr{A}_3^{\varepsilon, M}(\Pi)$, then 
\begin{align}\label{L1-cont}
\left(\E| \widetilde{\Psi}_{g_1} - \widetilde{\Psi}_{g_2} |  \right)^2 \le \E |\widetilde{\Psi}_{g_2}|^2  \cdot\bigg[\exp \Big(C_{\varepsilon, M} \big(L(g_1/g_2) + V(g_1/g_2)\big)  \Big) - 1 \bigg].
\end{align}
\end{prop}

\begin{proof}
Let $g_1, g_2 \in \mathscr{A}_3^{\varepsilon, M}(\Pi)$. Set $g: = (g_1/g_2)^2 $. Applying Proposition \ref{general-reg} to the function $g$ yields
$$
\E \widetilde{\Psi}_{g} \le \exp\Big(C_{\varepsilon,M} \Big(L(g) + V(g)\Big) \Big) \le \exp\Big(C'_{\varepsilon,M} \Big(L(g_1/g_2) + V(g_1/g_2)\Big) \Big).
$$
By multiplicativity, we have 
\begin{align*}
\E| \widetilde{\Psi}_{g_1} - \widetilde{\Psi}_{g_2} |   & = \E \Big(| \widetilde{\Psi}_{g_1/g_2} -1 |  | \widetilde{\Psi}_{g_2} | \Big) \le \Big(\E | \widetilde{\Psi}_{g_2} |^2\Big)^{\frac{1}{2}} \Big( \E|\widetilde{\Psi}_{g_1/g_2} -1|^2\Big)^{\frac{1}{2}}.
\end{align*}
By Jensen's inequality  
$$
\E \widetilde{\Psi}_{g_1/g_2} = \E (\exp(\overline{S}_{\log (g_1/g_2)})  \ge   \exp[\E(\overline{S}_{\log (g_1/g_2)}) ] = 1. 
$$ 
It follows that
$$
 \E|\widetilde{\Psi}_{g_1/g_2} -1|^2 \le \E |\widetilde{\Psi}_{g_1/g_2}|^2 -1  = \E\widetilde{\Psi}_{g} -1.
$$
Combining the above inequalities, we obtain Proposition \ref{L1-prop}.
\end{proof}

Slightly abusing notation, we keep the notation $\mathscr{T}$  for the induced topology defined by \eqref{tau-topo} on $\mathscr{A}_3^{\varepsilon, M}(\Pi)$.  As an immediate consequence of Proposition \ref{L1-prop}, we have

\begin{prop}\label{2-maps}
The two mappings from $\mathscr{A}_3^{\varepsilon, M}(\Pi)$ to $L^1(\Conf(E), \PP_\Pi)$ defined by 
$$g \to \widetilde{\Psi}_g, \quad  g \to \overline{\Psi}_g $$
are continuous with respect to the topology $\mathscr{T}$ on $\mathscr{A}_3^{\varepsilon, M}(\Pi)$.
\end{prop}

\begin{proof}[Proof of Proposition \ref{technical-prop}]
The  proof follows the proof of Corollary 4.8 in \cite{BQS}. Indeed, let $g$ be a function such that $\sup_E | g(x) - 1| < 1$. Taking $g_n$ as in Lemma \ref{lem-cut}, we obtain the convergence of $\Pi^{g_n}$ to $\Pi^g$ in the space of locally trace class operators and hence the weak convergence of
$\PP_{\Pi^{g_n}}$ to $\PP_{\Pi^{g}}$ in the space of probability measures on $\Conf(E)$. By assumption, $g_n-1$ is compactly supported, so by Proposition 2.1 of \cite{Buf-inf}, we have
$$
\PP_{\Pi^{g_n}} = \overline{\Psi}_{g_n} \cdot \PP_\Pi.
$$
By Proposition \ref{2-maps}, $ \overline{\Psi}_{g_n} \to  \overline{\Psi}_{g}$ in $L^1(\Conf(E), \PP_\Pi)$, so we have
$$
\overline{\Psi}_{g_n} \cdot \PP_\Pi \to \overline{\Psi}_{g} \cdot \PP_\Pi
$$
weakly in the space of probability measures on $\Conf(E)$, whence  $\PP_{\Pi^g} = \overline{\Psi}_{g} \cdot \PP_\Pi.$
The proof Proposition \ref{technical-prop} is complete.
\end{proof}

\section{Appendix}\label{sec-app}

Our aim here is to show that Palm measures of different orders are {\it mutually singular} for a point process rigid in the sense of Ghosh \cite{Ghosh-sine}, Ghosh-Peres \cite{Ghosh-rigid}. 

Let $E$ be a complete metric space, and let $\PP$ be a probability measure on $\Conf(E)$ admitting 
correlation measures of all orders; the $k$-th correlation measure of $\PP$  is denoted by $\rho_k$.
Given  $B\subset E$ a bounded Borel subset, let $\F(E\setminus B)$ be the sigma-algebra generated by all events of the form $\{\#_C = n\}$ with 
$C\subset E\setminus B$ bounded and Borel, $n \in \N$, and let $\F^{\PP}(E\setminus B)$ be the completion of 
$\F(E\setminus B)$ with respect to $\PP$. We can canonically identify $\Conf(E)$ with $\Conf(B) \times \Conf( E \setminus B)$. Then in this identification,  the events in $\F(E \setminus B)$ have the form
$$
\Conf(B) \times A, 
$$ 
where $A \subset \Conf(E\setminus B)$ is a measurable subset. By definition, assume that $\mathscr{X} \in \F(E \setminus B)$, and let $(p_1, \dots, p_k) \in B^k$ be any $k$-tuple of distinct points, then  $\XX \in \mathscr{X}$ if and only if $\XX \cup \{ p_1, \dots, p_k\} \in \mathscr{X}$.
Recall that a point process with distribution $\PP$ on $\Conf(E)$ is said to be rigid if for any bounded Borel subset $B \subset E$, the function $\#_B$ is $\F^{\PP}(E\setminus B)$-measurable.

\begin{prop}\label{rigid-dis}
Let $B\subset E$ be a bounded Borel subset.  Assume that  the function $\#_B$ is $\F^{\PP}(E\setminus B)$-measurable. Then, for any $k,l\in {\mathbb N}$, $k\neq l$, for $\rho_k$-almost any $k$-tuple $(p_1, \dots, p_k) \in B^k$ and $\rho_l$-almost any $l$-tuple $(q_1, \dots, q_l) \in B^l$, the reduced Palm measures $\PP^{p_1, \dots, p_k}$ and $\PP^{q_1, \dots, q_l}$ are mutually singular.
\end{prop}

\begin{rem}
After our preprint had appeared, S. Ghosh   studied the connection between rigidity and Palm measures of point processes  in  his preprint {\it Palm measures and rigidity phenomena in point processes}, 
arxiv:1509.00898, and, in particular, proved that rigidity implies singularity of Palm measures of different orders. Furthermore,  under additional assumptions on the conditional measures with respect to fixed configuration outside a bounded set,   Ghosh proved the mutual absolute continuity between Palm measures of the same order, in particular, treating the case of zero sets of Gaussian Analytic Functions on the plane and other non-determinantal point processes.  In our situation, however, we do not see how to check the assumptions that Ghosh needs without going through our argument. 
\end{rem}

\begin{proof}[Proof of Proposition \ref{rigid-dis}] 
For a nonnegative integer $n$, let 
$$
{\mathscr C}_n=\{\XX\in \Conf(E): \#_B(\XX)=n\}.
$$ 
By assumption, the function $\#_B$ is $\F^{\PP}(E \setminus B)$-measurable. Take a sequence $\mathscr{X}_n$ of disjoint $\F(E\setminus B)$-measurable subsets of $\Conf(E)$ such that for any nonnegative integer $n$ we have
$$
\PP(\mathscr{X}_n\Delta \mathscr{C}_n)=0.
$$
Set 
$$
\mathscr Y=\bigcup\limits_{n\ge k} \mathscr{X}_n\cap \mathscr{C}_{n-k};
$$
$$
\mathscr Z=\bigcup\limits_{n\ge l} {\mathscr X}_n\cap {\mathscr C}_{n-l}.
$$
The sets $\mathscr{Y}$and $\mathscr{Z}$ are disjoint by construction. 

{\bf Claim:} For $\rho_k$-almost any $k$-tuple $(p_1, \dots, p_k)$ and $\rho_l$-almost any $l$-tuple $(q_1, \dots, q_l)$ we have 
$$
\PP^{p_1, \dots, p_k}(\mathscr{Y})=1, \quad \PP^{q_1, \dots, q_l}(\mathscr{Z})=1.
$$

Indeed,  by definition of reduced Palm measures \eqref{def-Palm}, for any non-negative Borel function $u: \Conf(E) \times E^k\rightarrow \R $, we have
\begin{align}\label{def-palm}
\begin{split}
& \int\limits_{\Conf(E)}  \sum_{z_1, \dots, z_k \in \ZZ}^{*} u(\ZZ; z_1, \dots, z_k) \PP(d \ZZ)  \\ = &   \int\limits_{E^k} \rho_k(dp_1 \dots d p_k) \!\int\limits_{\Conf(E)} \!  u (\XX \cup \{p_1, \dots, p_k\}; p_1, \dots, p_k)  \PP^{p_1, \dots, p_k}(d\XX),
\end{split}
\end{align}
where $\sum\limits^{*}$ denotes the sum over  $k$-tuples of distinct points $z_1, \dots, z_k$ in $\ZZ$. 

For any $n \ge k$, substituting the function 
$$
u_n(\ZZ; z_1, \dots, z_k)  =  \mathds{1}_{\mathscr{X}_n  \cap   \mathscr{C}_n  } (\ZZ) \cdot \mathds{1}_{B^k}(z_1, \dots, z_k)
$$
into \eqref{def-palm}, we get
\begin{align}\label{palm2}
\begin{split}
& \int\limits_{\Conf(E)} \mathds{1}_{\mathscr{X}_n \cap \mathscr{C}_n} (\ZZ)   \sum_{z_1, \dots, z_k \in Z}^{*} \mathds{1}_{B^k}(z_1, \dots, z_k) \PP(d\ZZ) \\  = & \int\limits_{B^k} \rho_k(dp_1 \dots d p_k) \int\limits_{\Conf(E)}  \mathds{1}_{\mathscr{X}_n \cap\mathscr{C}_n } (\XX \cup \{p_1, \dots, p_k\}) \PP^{p_1, \dots, p_k} (d\XX).
\end{split}
\end{align}
Recall that by construction, $\mathscr{X}_n \in \F(E\setminus B)$, hence  {\it for all}   $p_1, \dots, p_k \in B$, we have 
\begin{align*}
\mathds{1}_{\mathscr{X}_n \cap\mathscr{C}_n } (\XX \cup \{p_1, \dots, p_k\})  =&   \mathds{1}_{\mathscr{X}_n } (\XX \cup \{p_1, \dots, p_k\}) \cdot  \mathds{1}_{\mathscr{C}_n } (\XX \cup \{p_1, \dots, p_k\}) \\ = & \mathds{1}_{\mathscr{X}_n } (\XX ) \cdot  \mathds{1}_{\mathscr{C}_{n-k} } (\XX)  = \mathds{1}_{\mathscr{X}_n  \cap \mathscr{C}_{n-k}} (\XX ).
\end{align*}
Substituting the above equality into \eqref{palm2}, we get 
\begin{align}\label{palm3}
\begin{split}
& \int\limits_{\Conf(E)} \mathds{1}_{\mathscr{X}_n \cap \mathscr{C}_n} (\ZZ)  \sum_{z_1, \dots, z_k \in \ZZ}^{*} \mathds{1}_{B^k}(z_1, \dots, z_k)\PP(d\ZZ) \\  = & \int\limits_{B^k}   \PP^{p_1, \dots, p_k} (\mathscr{X}_n  \cap \mathscr{C}_{n-k})\rho_k(dp_1 \dots d p_k).
\end{split}
\end{align}
Summing up the terms on the left hand side of \eqref{palm3} for $n \ge k$, we obtain the expression
\begin{align}\label{lhs}
\begin{split}
 & \sum_{n = k}^\infty \int\limits_{\Conf(E)} \mathds{1}_{\mathscr{X}_n \cap \mathscr{C}_n} (\ZZ)  \sum_{z_1, \dots, z_k \in \ZZ}^{*} \mathds{1}_{B^k}(z_1, \dots, z_k) \PP(d\ZZ) \\ = & \sum_{n = k}^\infty \int\limits_{\Conf(E)} \mathds{1}_{\mathscr{C}_n} (\ZZ)   \sum_{z_1, \dots, z_k \in \ZZ}^{*} \mathds{1}_{B^k}(z_1, \dots, z_k) \PP(d\ZZ) \\ = & \sum_{n = 0}^\infty \int\limits_{\Conf(E)} \mathds{1}_{\mathscr{C}_n} (\ZZ)   \sum_{z_1, \dots, z_k \in \ZZ}^{*} \mathds{1}_{B^k}(z_1, \dots, z_k) \PP(d\ZZ) \\ = &  \int\limits_{\Conf(E)}  \sum_{z_1, \dots, z_k \in \ZZ}^{*} \mathds{1}_{B^k}(z_1, \dots, z_k)\PP(d\ZZ) \\ = & \int\limits_{E^k} \mathds{1}_{B^k} (p_1, \dots, p_k) \rho_k(dp_1 \dots dp_k)  = \rho_k(B^k),
\end{split}
\end{align}
where we used the fact that if $n = 0, \dots, k-1$,  then 
$$
\forall \ZZ \in \mathscr{C}_n, \sum_{z_1, \dots, z_k \in \ZZ}^{*} \mathds{1}_{B^k}(z_1, \dots, z_k) =0.
$$
Similarly, summing up the terms on the right hand side of \eqref{palm3} for $n \ge k$, we obtain the expression
\begin{align}\label{rhs}
\begin{split}
 & \sum_{n = k}^\infty\int\limits_{B^k}   \PP^{p_1, \dots, p_k} (\mathscr{X}_n  \cap \mathscr{C}_{n-k})\rho_k(dp_1 \dots d p_k) \\ = & \int\limits_{B^k}   \PP^{p_1, \dots, p_k} \left(  \bigcup_{n \ge k}\mathscr{X}_n  \cap \mathscr{C}_{n-k} \right)\rho_k(dp_1 \dots d p_k) =   \int\limits_{B^k}   \PP^{p_1, \dots, p_k} \left( \mathscr{Y} \right)\rho_k(dp_1 \dots d p_k) .
\end{split}
\end{align}
By \eqref{palm3}, 
\begin{align}\label{palm4}
 \rho_k(B^k) = \int\limits_{B^k}   \PP^{p_1, \dots, p_k} \left( \mathscr{Y} \right)\rho_k(dp_1 \dots d p_k).
\end{align}
The equality \eqref{palm4} immediately implies that 
\begin{align*}
\PP^{p_1, \dots, p_k}(\mathscr{Y})= 1, \text{\, for $\rho_k$-almost any $k$-tuple $(p_1, \dots, p_k)\in B^k$}.
\end{align*}
The same argument yields that 
\begin{align*}
\PP^{q_1, \dots, q_l}(\mathscr{Z})= 1, \text{\, for $\rho_l$-almost any $l$-tuple $(q_1, \dots, q_l)\in B^l$}.
\end{align*}
The claim is proved, and Proposition \ref{rigid-dis} is proved completely.
\end{proof}

\section*{Acknowlegements}
We are deeply grateful to Alexander Borichev and Alexei Klimenko for  useful discussions. 
We are deeply grateful to the anonymous referees for their careful reading of the manuscript and 
many very helpful comments concerning presentation.
The authors were supported by A*MIDEX project (No. ANR-11-IDEX-0001-02), financed by Programme ``Investissements d'Avenir'' of the Government of the French Republic managed by the French National Research Agency (ANR).
The research of A. Bufetov on this project has received funding from the European Research Council (ERC) under the European Union's Horizon 2020 research and innovation programme under grant agreement No 647133 (ICHAOS). It has also been funded by the Grant MD 5991.2016.1 of the President of the Russian Federation, by  the Russian Academic Excellence Project `5-100' and by the Chaire Gabriel Lam\'e at the Chebyshev Laboratory.
 Y. Qiu is supported in part by the ANR grant 2011-BS01-00801.


\end{document}